\documentclass[11pt]{amsart}
\usepackage[T1]{fontenc}
\usepackage{lmodern}
\usepackage[a4paper,margin=26mm]{geometry}
\usepackage{amsmath,amssymb,amsthm,mathtools}
\usepackage{microtype}
\usepackage[linktocpage=true,hidelinks]{hyperref}
\newtheorem{theorem}{Theorem}[section]
\newtheorem{proposition}[theorem]{Proposition}
\newtheorem{lemma}[theorem]{Lemma}
\newtheorem{corollary}[theorem]{Corollary}
\theoremstyle{definition}

\newtheorem{example}[theorem]{Example}
\theoremstyle{remark}
\newtheorem{remark}[theorem]{Remark}
\numberwithin{equation}{section}
\newcommand{\Db}{\mathrm D^{\mathrm b}}
\newcommand{\Aut}{\operatorname{Aut}}
\newcommand{\Pic}{\operatorname{Pic}}
\newcommand{\End}{\operatorname{End}}
\newcommand{\Hom}{\operatorname{Hom}}
\newcommand{\Supp}{\operatorname{Supp}}
\newcommand{\id}{\operatorname{id}}
\newcommand{\Alb}{\operatorname{Alb}}
\newcommand{\SL}{\operatorname{SL}}

\newcommand{\C}{\mathbb C}
\newcommand{\Kcal}{\mathcal K}
\newcommand{\Z}{\mathbb Z}
\newcommand{\cO}{\mathcal O}
\newcommand{\bo}{\boxtimes}
\newcommand{\NS}{\operatorname{NS}}
\newcommand{\mat}[4]{\begin{pmatrix}#1&#2\\#3&#4\end{pmatrix}}
\title{Fourier--Mukai partners and autoequivalences\\of moduli spaces of vector bundles}
\author{Haotian Zuo}
\address{Department of Mathematics, University of Houston, Houston 77054, Texas, U.S.A}
\email{hzuo@cougarnet.uh.edu}
\date{}
\hypersetup{
 pdftitle={Fourier-Mukai partners and autoequivalences of moduli spaces of vector bundles},
 pdfauthor={Haotian Zuo}
}

\begin{document}
\begin{abstract}
We classify the smooth projective Fourier--Mukai partners and determine
the derived autoequivalence group of the moduli space $U_C(r,d)$
of stable bundles of coprime rank $r\geq2$ and degree $d$ on a smooth
complex projective curve, assuming $g(C)\geq3$ and $\Aut(C)=1$.
When the Jacobian $J(C)$ has endomorphism ring $\Z$, the partners are indexed by
$(\Z/r\Z)^\times/\{\pm1\}$.
We also prove that two coprime moduli spaces of vector bundles over complex curves of
genus at least four are derived equivalent if and only if they are
isomorphic.
\end{abstract}
\maketitle

\section{Introduction}\label{secmodulisetup}

Throughout, varieties are defined over $\C$, and derived equivalences
are $\C$-linear and exact. We write $\Db(X)$ for the bounded derived
category of coherent sheaves and $\operatorname{FM}(X)$ for the
isomorphism classes of smooth connected projective varieties $Y$
with $\Db(Y)\simeq\Db(X)$; these are the Fourier--Mukai partners of
$X$ \cite{Uehara}.

A smooth projective variety with ample canonical or anticanonical
bundle is determined by its derived category, and its autoequivalences
are generated by automorphisms, line bundles, and shifts \cite{BO}.
For abelian varieties, derived equivalence is instead governed by
isometries between products with their duals \cite{Orlov}. Moduli
spaces of vector bundles on curves combine these two geometries
through a finite quotient, and the quotient action imposes an
additional obstruction to derived equivalence.

Let $C$ be a smooth connected projective curve of genus $g\geq2$,
let $r\geq2$ and $d\in\Z$ be coprime, and fix $L\in\Pic^d(C)$.
Set
\[
 M=U_C(r,d),\qquad N=SU_C(r,L),\qquad J=\Pic^0(C),\qquad H=J[r].
\]
Here $M$ parametrizes stable bundles of rank $r$ and degree $d$,
and $N$ is the fibre over $L$ of the determinant morphism
$M\to\Pic^d(C)$. Both spaces are smooth and projective, and $N$
is Fano \cite{DN}. The tensor-product map is a finite \'etale
Galois cover
\[
 \pi:N\times J\longrightarrow M,\qquad (E,x)\longmapsto E\otimes x,
\]
with group $H$ acting by
$\eta\cdot(E,x)=(E\otimes\eta,x-\eta)$.
We use additive notation on abelian varieties; thus
$M\simeq(N\times J)/H$.

The geometry of $N$ already determines much of the bundle data.
Kouvidakis and Pantev computed automorphism groups and proved
Torelli results \cite{KP}; in genus at least four, its isomorphism
class determines the curve, rank, and degree modulo the rank up to
sign \cite[Corollary~2.12]{AB}. Work on semiorthogonal decompositions
and embeddings has also clarified the structure of $\Db(N)$
\cite{TT24,Tev,LM25}, while its autoequivalences follow from
Bondal--Orlov reconstruction and the automorphisms of $N$ \cite{Zuo}.
Our problem is to determine which data survive when the determinant
varies and the finite quotient couples $N$ to $J$.

The reconstruction and abelian equivalence theorems do not by themselves
control compatibility with this finite quotient. Existing lifting and
descent criteria \cite{BM,LP,KS} and the obstruction theory for linearizing
invariant kernels \cite{Ploog} provide the categorical framework.
Applying this framework here requires identifying arbitrary partners
as product quotients with the same Fano factor and establishing
compatibility with the covers and the torsion linearization criterion.
The lifting and product-decomposition results also apply beyond these
moduli spaces (Theorems~\ref{liftmain} and~\ref{prodmain}).

For $g\geq3$ and $\Aut(C)=1$, we classify all smooth projective
Fourier--Mukai partners of $M$ and determine $\Aut\Db(M)$.
The classification retains both an abelian variety and an
identification of its $r$-torsion with $J[r]$. The condition on
these data is stronger than invariance of an equivalence kernel:
the kernel must admit a compatible linearization.

For abelian varieties $A,B$, let $\Hom(A,B)$ denote the group of
homomorphisms of algebraic groups and put $\End(A)=\Hom(A,A)$.
Write $\widehat A=\Pic^0(A)$ for the dual abelian variety and
$\widehat u:\widehat B\to\widehat A$ for the dual of $u:A\to B$.
An \emph{Orlov isometry} is a group isomorphism
\begin{equation}\label{orlovisometries}
 f=\mat{a}{b}{c}{e}:A\times\widehat A\xrightarrow{\sim}B\times\widehat B,
 \qquad f^{-1}=\mat{\widehat e}{-\widehat b}{-\widehat c}{\widehat a},
\end{equation}
using the canonical identification $\widehat{\widehat A}=A$.
Such isometries characterize derived equivalences of abelian
varieties \cite{Orlov}.

For an abelian variety $A$ and a group isomorphism
$\beta:J[r]\xrightarrow{\sim}A[r]$, set
\begin{equation}\label{generalpartnerquotient}
 Y_{A,\beta}=(N\times A)/H,\qquad
 \eta\cdot(E,x)=(E\otimes\eta,x-\beta(\eta)).
\end{equation}
The action is free, so $Y_{A,\beta}$ is smooth and projective.
The classification is as follows.
\begin{theorem}\label{partner-pair-classification}
Assume $g\geq3$ and $\Aut(C)=1$.
The Fourier--Mukai partners of $M$ are precisely the quotients
$Y_{A,\beta}$ for which there exists an Orlov isometry
\begin{equation}\label{partnerisometrycriterion}
 \mat{a}{b}{c}{e}:J\times\widehat J\xrightarrow{\sim}A\times\widehat A,
 \qquad a|_{J[r]}=\beta,\qquad c\in r^2\Hom(J,\widehat A).
\end{equation}
Two such quotients $Y_{A,\beta}$ and $Y_{A',\beta'}$ are isomorphic
if and only if there is a group isomorphism $u:A\to A'$ with
$u\beta=\beta'$. In particular, $\operatorname{FM}(M)$ is finite.
\end{theorem}
Divisibility in \eqref{partnerisometrycriterion} is in the integral
group $\Hom(J,\widehat A)$. No restriction on $\End(J)$ is imposed.

When $\End(J)=\Z$, it has the following explicit form.
Write $\varphi(r)=|(\Z/r\Z)^\times|$. For $k\in(\Z/r\Z)^\times$, put
\begin{equation}\label{partnerquotient}
 X_k=Y_{J,[k]},\qquad [k](\eta)=k\eta.
\end{equation}
Thus $X_1\simeq M$.

\begin{corollary}\label{complete-scalar-partners}
Assume $g\geq3$, $\Aut(C)=1$, and $\End(J)=\Z$.
Every Fourier--Mukai partner of $M$ is isomorphic to an $X_k$, and
\(X_k\simeq X_l\text{ if and only if }l\equiv\pm k\pmod r.\)
Consequently, $\#\operatorname{FM}(M)=1$ for $r=2$ and
$\#\operatorname{FM}(M)=\varphi(r)/2$ for $r\geq3$.
\end{corollary}

Taking $d=1$ and varying $r$ therefore gives moduli spaces with
arbitrarily many pairwise nonisomorphic Fourier--Mukai partners,
despite the Fano property of their fixed-determinant fibres.

Let $\Aut\Db(M)$ denote the group of autoequivalences of the derived category $\Db(M)$, up to natural isomorphism,
with composition as the group law. Write $U(J\times\widehat J)$
for the group of Orlov self-isometries and set
\begin{equation}\label{generalcongruence}
 U(J\times\widehat J;r)=\{v\in U(J\times\widehat J):
 v|_{(J\times\widehat J)[r]}=\id\}.
\end{equation}
Also put $\Gamma(r)=\ker(\SL_2(\Z)\to\SL_2(\Z/r\Z))$.

\begin{theorem}\label{fullgroup}
\label{groupgeneral}\label{grpExact}
Assume $g\geq3$ and $\Aut(C)=1$. There is an exact sequence
\begin{equation}\label{grpExactSequence}
 1\longrightarrow\Z^2\times J(\C)\times\widehat J(\C)
 \longrightarrow\Aut\Db(M)
 \longrightarrow U(J\times\widehat J;r)\longrightarrow1.
\end{equation}
The image of the $\Z^2$ factor in $\Aut\Db(M)$ is central, and the conjugation action on
$J(\C)\times\widehat J(\C)$ is the natural action of
$U(J\times\widehat J;r)$.
If $\End(J)=\Z$, the sequence splits and yields a noncanonical
isomorphism
\begin{equation}\label{fullformula}
 \Aut\Db(M)\simeq
 \Z^2\times\bigl((J(\C)\times\widehat J(\C))\rtimes\Gamma(r)\bigr).
\end{equation}
Here $\Gamma(r)$ acts by matrix multiplication after identifying
$\widehat J$ with $J$ by the canonical principal polarization.
\end{theorem}

The maps in \eqref{grpExactSequence} and the induced conjugation
action on its kernel are described in Section~\ref{secgroupsequence}.

Derived equivalence between two bundle moduli spaces is more rigid.
The following theorem requires no restriction on the automorphism
groups or Jacobian endomorphism rings of the curves.

\begin{theorem}[Derived Torelli for $U_C(r,d)$]
\label{moduli-derived-classification}
Let $C,C'$ be smooth connected projective curves with
$g(C),g(C')\geq4$.
Let $r,r'\geq2$ and $d,d'\in\Z$ satisfy
$\gcd(r,d)=\gcd(r',d')=1$.
The following conditions are equivalent:
\begin{enumerate}
\item $\Db(U_C(r,d))\simeq\Db(U_{C'}(r',d'))$;
\item $U_C(r,d)\simeq U_{C'}(r',d')$;
\item $C\simeq C'$, $r=r'$, and $d'\equiv\pm d\pmod r$.
\end{enumerate}
\end{theorem}

The paper is organized as follows. Section~\ref{secgeometry} studies
the geometry of the finite quotient presentation, including torsion
fixed loci, Albanese morphisms, and anticanonical models. It proves the
derived Torelli theorem (Theorem~\ref{moduli-derived-classification})
and identifies the product-quotient structure of arbitrary partners
(Theorem~\ref{classification-geometric-partner}).
Section~\ref{seclifting} develops the results on equivalences needed
for the classification: equivariant lifting
(Theorem~\ref{liftmain}), decomposition of equivalences between
Fano--abelian products (Theorem~\ref{prodmain}), and the criterion for
linearization under torsion translations. These results are applied
in Section~\ref{secpartners} to prove the complete partner classification
(Theorem~\ref{partner-pair-classification}) and its scalar specialization
(Corollary~\ref{complete-scalar-partners}). That section also studies
explicit scalar quotients and gives real-multiplication examples
illustrating the effect of additional endomorphisms on the classification.
Finally, Section~\ref{secmoduli} determines the autoequivalence group
and proves Theorem~\ref{fullgroup}: it first establishes the general
exact sequence and then constructs a splitting when $\End(J)=\Z$.

\section{Geometry of the moduli spaces and their partners}\label{secgeometry}
The determinant morphism and the anticanonical contraction separate
the two factors of the cover $N\times J\to M$. We first recover the
Fano factor and its torsion action from the anticanonical model. We
then use an Albanese splitting to recover the product-quotient
structure of an arbitrary Fourier--Mukai partner.

We write $\NS(V)=\Pic(V)/\Pic^0(V)$, $\Alb(V)$ for the Albanese
variety, and $\Aut^0(V)$ for the identity component of its automorphism
group. For a line bundle $R$, put $M_R=(-\otimes R)$; $k(v)$ denotes
the skyscraper sheaf at a closed point, in degree zero. Set
\[
 \Supp(K)=\bigcup_i\Supp\mathcal H^i(K),\qquad
 R(V,D)=\bigoplus_{n\geq0}H^0(V,\cO_V(nD)).
\]
Operations on complexes are derived unless ordinary sheaf operations
are specified.

An equivalence between smooth projective varieties has a
Fourier--Mukai kernel, unique up to isomorphism
\cite[Theorem~5.14]{Huybrechts}. Our conventions are
\[
 \Phi_K(E)=Rp_{2*}(p_1^*E\otimes K),\qquad
 \Phi_{K_2\star K_1}\simeq\Phi_{K_2}\Phi_{K_1},
\]
\[
 K_2\star K_1=Rp_{13*}(p_{12}^*K_1\otimes p_{23}^*K_2),
\]
with projections from the relevant products. On nonproper varieties,
we use kernels whose support is proper over both factors. Hom spaces
of kernels are taken in the derived category of the product;
$\End(K)=\Hom(K,K)$, and $K$ is simple if this algebra is $\C$.
Writing $\Delta_V:V\hookrightarrow V^2$ for the diagonal, the identity
kernel is $\Delta_{V*}\cO_V$. For $K\in\Db(X_1\times X_2)$ and
$L\in\Db(Y_1\times Y_2)$, we use
\[
 K\bo L=p_{13}^*K\otimes p_{24}^*L
 \in\Db\bigl((X_1\times Y_1)\times(X_2\times Y_2)\bigr),
\]
and write $\Phi_K\bo\Phi_L$ for the resulting functor.

\subsection{Fixed determinants and torsion actions}

Keep $M=U_C(r,d)$, $N=SU_C(r,L)$, $J=\Pic^0(C)$, and $H=J[r]$
from Section~\ref{secmodulisetup}; here $g\geq2$, $r\geq2$, and
$\gcd(r,d)=1$. The varieties $M,N$ are smooth, connected, and
projective, and $\dim N=(r^2-1)(g-1)$. Moreover,
\begin{equation}\label{modpicard}
 \Pic(N)=\Z[\Theta_N],\qquad \omega_N\simeq\Theta_N^{-2},
\end{equation}
where $\Theta_N$ is ample \cite[Theorems B and F]{DN}.
The choice of $L$ gives $\delta:M\to J$,
$E\mapsto\det E\otimes L^{-1}$. Pulling back $[r]:J\to J$
along $\delta$ gives the finite \'etale Galois cover
\begin{equation}\label{cover}
 \pi:N\times J\longrightarrow M,\qquad (E,x)\longmapsto E\otimes x,
 \qquad \delta\pi=[r]\operatorname{pr}_J,
\end{equation}
with group $H$ acting by
\begin{equation}\label{actionH}
 h_\eta(E,x)=(E\otimes\eta,x-\eta).
\end{equation}
Let $N^\eta$ be the fixed locus of tensoring by $\eta\in H$.

\begin{lemma}[Fixed loci and quotient singularities]
\label{modfixed}\label{classification-terminal}
Let $\eta\in H$ have order $n>1$. Every nonempty component of
$N^\eta$ has dimension $(r^2/n-1)(g-1)$, and
\[
 \operatorname{codim}_N N^\eta
 \geq r^2(1-1/n)(g-1)\geq2.
\]
For every subgroup $K\subset H$, the quotient $N/K$ is
$\mathbb Q$-factorial and canonical. It is terminal if
$(g,r)\ne(2,2)$.
\end{lemma}
\begin{proof}
We compute the tangent representation using the cyclic cover of $C$
associated with $\eta$, as in \cite[Lemma~2.1 and (2.16)]{BH}.
The calculation below also covers $(g,r)=(2,2)$, which is excluded
in that reference.

For $[E]\in N^\eta$, choose $\tau:\eta^{\otimes n}\simeq\cO_C$
and $\varphi:E\simeq E\otimes\eta$. Simplicity of $E$ allows us
to rescale $\varphi$ so that its $n$-fold composite, identified using
$\tau$, is $\id_E$. The associated connected cyclic \'etale cover is
\[
 p:C_\eta=\operatorname{Spec}_C
 \left(\bigoplus_{j=0}^{n-1}\eta^{-j}\right)\longrightarrow C.
\]
Its tautological section $s$ of $p^*\eta$ satisfies
$p^*\tau(s^n)=1$. Write $p^*\varphi=T\otimes s$ and put
$\zeta=\exp(2\pi i/n)$. Then
\[
 T^n=\id,\qquad p^*E=\bigoplus_{a=0}^{n-1}V_a,
 \qquad V_a=\ker(T-\zeta^a\id).
\]
For the deck generator with $\gamma^*s=\zeta s$, one has
$\gamma^*T=\zeta^{-1}T$, hence $\gamma^*V_a\simeq V_{a+1}$.
All these subbundles have rank $r/n$.

Let $\mathcal B_j$ be the $\zeta^j$-eigensubbundle of
$\mathcal{E}nd(E)$ for
$\alpha(u)=\varphi^{-1}(u\otimes\id_\eta)\varphi$.
Since $p^*\alpha(u)=T^{-1}uT$,
\[
 p^*\mathcal B_j\simeq
 \bigoplus_{a=0}^{n-1}\mathcal{H}om(V_a,V_{a-j}),\qquad
 \operatorname{rk}\mathcal B_j=\frac{r^2}{n}.
\]
The summands $\mathcal B_j$ have degree zero because
$\mathcal{E}nd(E)$ is semistable of degree zero. Its global sections
are the scalar endomorphisms, lying in $\mathcal B_0$. Removing
this scalar summand gives
\[
 \mathcal{E}nd_0(E)=\mathcal B_0^0\oplus
 \bigoplus_{j=1}^{n-1}\mathcal B_j,
 \qquad \mathcal B_0^0=\ker(\operatorname{tr}|_{\mathcal B_0}),
\]
whose summands have degree zero and no global sections.
The tangent action is induced by $\alpha$ on
$T_{[E]}N=H^1(C,\mathcal{E}nd_0(E))$. Riemann--Roch gives
the multiplicities
\[
 m_0=\left(\frac{r^2}{n}-1\right)(g-1),\qquad
 m_j=\frac{r^2}{n}(g-1)\quad(1\leq j<n).
\]
The fixed locus is smooth with invariant tangent space, proving its
dimension and codimension. The age is
\begin{equation}\label{modage}
 \operatorname{age}(\eta;T_{[E]}N)
 =\sum_{j=1}^{n-1}\frac jn m_j
 =\frac{r^2(n-1)(g-1)}{2n}.
\end{equation}
It is at least $r^2(g-1)/4\geq1$, and is strictly greater than one
unless $(g,r)=(2,2)$.

For any $K\subset H$, these bounds apply to every nonidentity
element of every stabilizer $K_{[E]}$. The codimension bound
excludes pseudoreflections. Analytic linearization and the
Reid--Tai criterion \cite[Theorem~2.3(ii), (iii)]{ReidTai} therefore
give the asserted canonical and terminal singularities.

Finally, a finite quotient of a smooth variety is
$\mathbb Q$-factorial. In the present case, for $q:N\to N/K$ and
a Weil divisor $D$ on $N/K$, the Cartier divisor $q^*D$ is principal
on the semilocal scheme $N\times_{N/K}\operatorname{Spec}\cO_{N/K,z}$
for each $z\in N/K$. The norm of a rational equation has divisor
$q_*q^*D=|K|D$. Thus $|K|D$ is locally Cartier.
\end{proof}

If $g\geq3$ and $\Aut(C)=1$, \cite[Theorem A]{KP} gives
\begin{equation}\label{fixedaut}
 \Aut(N)=\{s_\eta:\eta\in H\},\qquad s_\eta(E)=E\otimes\eta.
\end{equation}
For $r\geq3$, coprimality implies $r\nmid2d$, excluding
fixed-determinant dualization; for $r=2$, the map
$E\mapsto E^\vee\otimes L$ is the identity on $N$.
Lemma~\ref{modfixed} makes the tensoring action faithful.

\subsection{Albanese morphisms and anticanonical models}

\begin{lemma}[Albanese morphism of a finite product quotient]
\label{quotient-albanese}
Let $B$ be a smooth connected projective variety with $\Alb(B)=0$,
let $A$ be an abelian variety, and let a finite abelian group $H$
act on $B$ and embed in $A$ by $\iota:H\hookrightarrow A(\C)$.
Put $X=(B\times A)/H$ for the action
$h(b,a)=(hb,a+\iota(h))$, and let $q:A\to\bar A=A/\iota(H)$.
Then $\delta:X\to\bar A$, $[b,a]\mapsto q(a)$, is an Albanese
morphism, and
\[
 B\times A\xrightarrow{\sim}X\times_{\bar A}A.
\]
In particular, $\delta$ is smooth and proper with connected fibres,
$\delta_*\cO_X=\cO_{\bar A}$, and
$\delta^*:\Pic^0(\bar A)\xrightarrow{\sim}\Pic^0(X)$.
\end{lemma}
\begin{proof}
The action on $B\times A$ is free. The map
$(b,a)\mapsto([b,a],a)$ identifies two finite \'etale $H$-torsors
over $X$, and is therefore an isomorphism. Smoothness, properness,
and connectedness of the fibres follow after base change to $A$;
Stein factorization gives $\delta_*\cO_X=\cO_{\bar A}$.

Choose $x_0=[b_0,0]$. A pointed morphism $X\to T$ to an abelian
variety pulls back to a morphism $B\times A\to T$ constant on
every $B$-fibre, since $\Alb(B)=0$. It thus comes from a
homomorphism $A\to T$. Its $H$-invariance is exactly the condition
that this homomorphism factor through $\bar A$. This proves the
Albanese assertion, and Albanese--Picard duality gives the Picard
isomorphism \cite[Remarks~5.24--5.25]{KleimanPicard}.
\end{proof}

For the rest of this subsection, let $B$ be a smooth connected
projective Fano variety and $A$ an abelian variety. Let a finite
abelian group $H$ act on $B$ by $h_B$ and embed in $A$ by $\iota$.
Assume
\begin{equation}\label{smallhyp}
 \operatorname{codim}_B B^h\geq2\qquad(h\ne0),
\end{equation}
allowing empty fixed loci. In particular, the action on $B$ is
faithful. Set
\begin{equation}\label{generalaction}
 Y=B\times A,\qquad
 h_Y(b,a)=(h_B(b),a+\iota(h)),\qquad
 X=Y/H,\qquad \pi:Y\longrightarrow X.
\end{equation}
The map $\pi$ is finite \'etale and $X$ is smooth projective.
Kodaira vanishing gives $\Alb(B)=0$. Thus
Lemma~\ref{quotient-albanese} identifies the Albanese morphism
$\delta:X\to\bar A=A/\iota(H)$ and its Cartesian pullback
$Y=X\times_{\bar A}A$.

For $q:A\to\bar A$, write $L_\chi\in\Pic^0(\bar A)$ for the
eigensheaf of $q_*\cO_A$ indexed by
$\chi\in H^\vee=\Hom(H,\C^*)$. Flat base change gives
\begin{equation}\label{characters}
 \pi_*\cO_Y=\bigoplus_{\alpha\in T}\alpha,
 \qquad T=\{\delta^*L_\chi:\chi\in H^\vee\}\subset\Pic^0(X).
\end{equation}
These \emph{character line bundles} form the subgroup
$T=\ker(\pi^*:\Pic^0(X)\to\Pic^0(Y))$ of order $|H|$.

Write $a:B\to Z=B/H$ for the quotient and
$p:X\to Z$ for the morphism induced by $a\operatorname{pr}_B$.
Let $Z_{\mathrm{reg}}$ be the smooth locus and
$B^\circ=B\setminus\bigcup_{h\ne0}B^h$ the free locus.

\begin{lemma}\label{small}
We have $a^{-1}(Z_{\mathrm{reg}})=B^\circ$, and
$a:B^\circ\to Z_{\mathrm{reg}}$ is a universal finite \'etale
cover with group $H$. Every automorphism of $Z$ lifts to an
automorphism of $B$, uniquely up to composition with an element of $H$.
\end{lemma}
\begin{proof}
The free quotient $a(B^\circ)$ is smooth. Conversely,
$a^{-1}(Z_{\mathrm{reg}})\to Z_{\mathrm{reg}}$ is \'etale in
codimension one by \eqref{smallhyp}, hence \'etale by purity
\cite[Tag~0BMB]{Stacks}. Its fibres have $|H|$ points, so the
stabilizers over $Z_{\mathrm{reg}}$ are trivial.

Every connected finite \'etale cover $B'\to B$ is Fano, and
Kodaira vanishing and Hirzebruch--Riemann--Roch give
\[
 1=\chi(B',\cO_{B'})
 =\deg(B'/B)\chi(B,\cO_B)=\deg(B'/B).
\]
A connected finite \'etale cover of $B^\circ$ extends, by
normalization and purity, to one of $B$, since
$\operatorname{codim}(B\setminus B^\circ)\geq2$. Thus
$B^\circ$ has no nontrivial connected finite \'etale cover,
proving universality.

An automorphism of $Z$ lifts over $Z_{\mathrm{reg}}$ to this
universal cover, uniquely up to $H$. Since $B$ is the normalization
of $Z$ in $\C(B^\circ)$, this lift and its inverse extend by
functoriality of normalization \cite[Tag~035J]{Stacks}. The
uniqueness relation extends from the dense open set $B^\circ$.
\end{proof}

In particular, $a$ is quasi-\'etale. For $b\in B$, write
$H_b\subset H$ for its stabilizer.

\begin{proposition}[The anticanonical contraction]
\label{quotient-anticanonical}
There is an isomorphism of graded rings
\begin{equation}\label{canring}
 R(X,-K_X)\simeq R(B,-K_B)^H.
\end{equation}
The line bundle $\omega_X^{-1}$ is semiample, and
$p:X\to Z$ is its anticanonical contraction. Moreover,
\[
 \bigl(p^{-1}(a(b))\bigr)_{\mathrm{red}}\simeq A/\iota(H_b)
 \qquad(b\in B).
\]
In particular, $p$ has connected fibres.
\end{proposition}
\begin{proof}
A translation-invariant top form on $A$ identifies
$\pi^*\omega_X$ with $\operatorname{pr}_B^*\omega_B$
$H$-equivariantly, using the differential action on $\omega_B$.
Taking sections and invariants gives \eqref{canring}, since
$H^0(A,\cO_A)=\C$.

Choose $m>0$ with $\omega_B^{-m}$ globally generated. For any
finite $H$-orbit, choose a section $s$ nonvanishing on that orbit.
The product $\prod_{h\in H}h\cdot s$ is an invariant section of
$\omega_B^{-m|H|}$, nonvanishing there. These sections descend to
$X$ and show that $\omega_X^{-m|H|}$ is globally generated.
Since $-K_B$ is ample,
\[
 \operatorname{Proj}R(X,-K_X)
 \simeq\operatorname{Proj}\bigl(R(B,-K_B)^H\bigr)=B/H.
\]
The associated map pulls back along $\pi$ to
$a\operatorname{pr}_B$, so it is $p$.

Because $\pi$ is \'etale, the inverse image of the reduced fibre
over $a(b)$ is $(H\cdot b)\times A$ as a reduced scheme.
Its quotient by $H$ is $A/\iota(H_b)$, giving the fibre formula
and connectedness.
\end{proof}

\subsection{Reconstruction under derived equivalence}

Let $P\in\Db(W_1\times W_2)$ induce an equivalence between smooth
connected projective varieties. Put $R_i=R(W_i,-K_{W_i})$ and,
when $-K_{W_i}$ is semiample, let
$p_i:W_i\to Z_i=\operatorname{Proj}R_i$ be its contraction.
We use compatibility with canonical sections, as well as with the
Serre functor, to identify these maps.

\begin{lemma}[Global generation of canonical powers]
\label{canonical-global-generation}
For every $m\in\Z$, the line bundle $\omega_{W_1}^{\otimes m}$ is
globally generated if and only if $\omega_{W_2}^{\otimes m}$ is.
In particular, $-K_{W_1}$ is semiample if and only if $-K_{W_2}$ is.
\end{lemma}
\begin{proof}
By \cite[Lemma~3.3]{LO}, conjugation by $\Phi_P$ identifies the
functors $-\otimes\omega_{W_i}^n$ and their morphisms induced by
sections, for every $n\in\Z$. We may choose the resulting maps
\[
 \rho_n:H^0(W_1,\omega_{W_1}^n)
 \xrightarrow{\sim}H^0(W_2,\omega_{W_2}^n)
\]
compatibly with multiplication. In particular, multiplication by
$s$ on $E$ corresponds to multiplication by $\rho_n(s)$ on
$\Phi_P(E)$.

Suppose $s_1,\ldots,s_\ell$ generate $\omega_{W_1}^m$.
If all $\rho_m(s_j)$ vanished at $w\in W_2$, their multiplication
maps on $k(w)$ would vanish. Hence every $s_j$ would act by zero
on the nonzero object $E=\Phi_P^{-1}(k(w))$. On the open set
where $s_j$ is nonvanishing this map is invertible, so $E$ would
vanish there. These opens cover $W_1$, a contradiction. The
inverse equivalence gives the converse.
\end{proof}

\begin{lemma}\label{antican}
Assume $-K_{W_1}$ is semiample. Then $\Phi_P$ induces a graded-ring
isomorphism $R_1\simeq R_2$ and a corresponding isomorphism
$\sigma:Z_1\xrightarrow{\sim}Z_2$ satisfying
\begin{equation}\label{supportbase}
 \Supp(P)\subset
 \{(u,v)\in W_1\times W_2:p_2(v)=\sigma(p_1(u))\}.
\end{equation}
\end{lemma}
\begin{proof}
The preceding lemma makes $-K_{W_2}$ semiample, and the compatible
maps $\rho_{-k}$ give $\tau:R_1\xrightarrow{\sim}R_2$.
Set $\sigma=\operatorname{Proj}(\tau^{-1})$. Choose $N>0$
sufficiently divisible that $\omega_{W_i}^{-N}$ is generated and
its complete linear system is $p_i$ followed by a projective
embedding. Write $V_i=H^0(W_i,\omega_{W_i}^{-N})$.

For $u\in W_1$, put $F=\Phi_P(k(u))$. A section $s\in V_1$
vanishing at $u$ acts by zero on $k(u)$, so $\rho_{-N}(s)$ acts
by zero on $F$. It must therefore vanish at every
$v\in\Supp(F)$: otherwise multiplication by it would be both
zero and invertible near $v$. The evaluation hyperplanes satisfy
\[
 \rho_{-N}(\ker\operatorname{ev}_{1,u})
 \subseteq\ker\operatorname{ev}_{2,v}.
\]
Both have codimension one, so equality holds, giving
$p_2(v)=\sigma(p_1(u))$.

For $i_u:W_2\hookrightarrow W_1\times W_2$, the kernel formula
and projection formula give $F\simeq Li_u^*P$. Since $P$ is
perfect, derived Nakayama identifies $\Supp(F)$ with the fibre
of $\Supp(P)$ over $u$ \cite[Tag~0BCD]{Stacks}. Thus the equality
of anticanonical images holds at every closed point of
$\Supp(P)$, proving \eqref{supportbase}.
\end{proof}

For $i=1,2$, let $B_i$ be smooth connected projective Fano varieties,
$A_i$ abelian varieties, and $H_i$ finite abelian groups acting
on $B_i$. Fix embeddings $\iota_i:H_i\hookrightarrow A_i(\C)$
and put
\[
 X_i=(B_i\times A_i)/H_i,\qquad
 h\cdot(b,a)=(h\cdot b,a+\iota_i(h)).
\]
Assume $\operatorname{codim}_{B_i}B_i^h\geq2$ for $h\ne0$.

\begin{proposition}[Equivariant reconstruction of the Fano factor]
\label{reconstruction}
If $\Db(X_1)\simeq\Db(X_2)$, there are an isomorphism
$f:B_1\xrightarrow{\sim}B_2$ and a group isomorphism
$\nu:H_1\xrightarrow{\sim}H_2$ with
\[
 f(h\cdot b)=\nu(h)\cdot f(b)
 \qquad(h\in H_1,\ b\in B_1).
\]
\end{proposition}
\begin{proof}
Proposition~\ref{quotient-anticanonical} and Lemma~\ref{antican}
give an isomorphism $\sigma:B_1/H_1\xrightarrow{\sim}B_2/H_2$.
By Lemma~\ref{small}, its restriction to the regular loci lifts
to an isomorphism $f^\circ:B_1^\circ\xrightarrow{\sim}B_2^\circ$
of their universal finite \'etale covers. Normalization extends
this to $f:B_1\xrightarrow{\sim}B_2$, as in the proof of that
lemma. Conjugation by $f^\circ$ identifies the deck groups, and
the resulting equivariance relation extends to $B_1$ by density.
\end{proof}

\begin{proof}[Proof of Theorem~\ref{moduli-derived-classification}]
Choose determinants $L,L'$ of degrees $d,d'$. The covers
\eqref{cover}, the Fano property \eqref{modpicard}, and
Lemma~\ref{modfixed} allow us to apply
Proposition~\ref{reconstruction} to a derived equivalence of the
two moduli spaces. It gives
$SU_C(r,L)\simeq SU_{C'}(r',L')$. Since both genera are at least
four, \cite[Corollary~2.12]{AB} gives $C\simeq C'$, $r=r'$,
and $d'\equiv\pm d\pmod r$.

Conversely, choose $\sigma:C'\xrightarrow{\sim}C$ and
$\epsilon\in\{1,-1\}$ with $d'\equiv\epsilon d\pmod r$.
For a line bundle $T$ on $C'$ of degree $(d'-\epsilon d)/r$,
the assignment
\[
 E\longmapsto
 \begin{cases}
  \sigma^*E\otimes T,&\epsilon=1,\\
  (\sigma^*E)^\vee\otimes T,&\epsilon=-1
 \end{cases}
\]
is an isomorphism of moduli spaces: it preserves stability and
has an inverse of the same form. An isomorphism in turn induces
a derived equivalence.
\end{proof}

\begin{corollary}[Recovery of the curve and rank]\label{moduli-torelli}
Let $C,C'$ be smooth connected projective curves of genus at least
two, and let $r,r'\geq2$ and $d,d'\in\Z$ satisfy
$\gcd(r,d)=\gcd(r',d')=1$. If
$\Db(U_C(r,d))\simeq\Db(U_{C'}(r',d'))$, then $C\simeq C'$
and $r=r'$.
\end{corollary}
\begin{proof}
Proposition~\ref{reconstruction} again gives
$SU_C(r,L)\simeq SU_{C'}(r',L')$.
The fixed-determinant Torelli theorem in \cite[Remark~5.3]{ABGM}
recovers the curve and rank, except when both genera and ranks
are two and both determinant degrees are even. Coprimality
excludes that exception.
\end{proof}

\subsection{The geometry of an arbitrary partner}

Let $B$ be a smooth connected projective Fano variety and $H$ a
finite abelian group acting on $B$ as in \eqref{smallhyp}.
Write $a:B\to Z=B/H$, and assume $Z$ is terminal.

\begin{lemma}\label{classification-ramification}\label{crepant-quotient}
Let $F$ be a smooth connected projective variety and
$q:F\to Z$ a surjective generically finite morphism satisfying
$K_F\sim_{\mathbb Q}q^*K_Z$. Then there is an isomorphism
$f:F\xrightarrow{\sim}B$ with $q=a\circ f$.
\end{lemma}
\begin{proof}
We first replace the assumed linear equivalence by equality of
compatible canonical divisors. Choose a resolution
$b:\widetilde Z\to Z$ and resolve $F\dashrightarrow\widetilde Z$
to obtain $\alpha:T\to F$ birational and
$\widetilde q:T\to\widetilde Z$ generically finite, with
$b\widetilde q=q\alpha$. Define all canonical divisors by
pulling back a rational top form on $Z$. Then
\[
 K_{\widetilde Z}=b^*K_Z+E_b,\qquad
 K_T=\widetilde q^*K_{\widetilde Z}+R
     =\alpha^*K_F+E_\alpha,
\]
where $E_b\geq0$ by terminality, $R\geq0$ is the ramification
divisor, and $E_\alpha$ is $\alpha$-exceptional. Pushing down
gives
\[
 K_F-q^*K_Z=\alpha_*(\widetilde q^*E_b+R)\geq0.
\]
An effective $\mathbb Q$-linearly trivial divisor on a projective
variety is zero, so $K_F=q^*K_Z$.

In the Stein factorization $F\xrightarrow{v}Z'\xrightarrow{u}Z$,
the map $v$ is projective birational and $u$ is finite, with $Z'$
normal. Since $v$ is an isomorphism over the codimension-one
points of $Z'$, pushing forward gives
\[
 K_{Z'}=u^*K_Z,\qquad K_F=v^*K_{Z'}.
\]
The first equality and the ramification formula show that $u$ is
unramified in codimension one. Purity makes
$u^{-1}(Z_{\mathrm{reg}})\to Z_{\mathrm{reg}}$ \'etale.
It is connected, and Lemma~\ref{small} identifies it with
$B^\circ/K$ for some subgroup $K\subset H$. Normalization then
gives $Z'\simeq B/K$ over $Z$.

The quotient $B/K$ is $\mathbb Q$-factorial by the local norm
argument used in Lemma~\ref{modfixed}. It is also terminal:
at $x\in B$, the stabilizer $K_x$ is a subgroup of $H_x$, and
\eqref{smallhyp} excludes pseudoreflections in both tangent
representations. Terminality of $B/H$ gives
$\operatorname{age}(h;T_xB)>1$ for every $0\ne h\in H_x$;
the same inequalities for $K_x$ give terminality of $B/K$
by \cite[Theorem~2.3(iii)]{ReidTai}.

These two properties have distinct roles in excluding a
nontrivial $v$. Terminality and $K_F=v^*K_{Z'}$ exclude
$v$-exceptional prime divisors. If $D$ is a $v$-ample Cartier
divisor, $\mathbb Q$-factoriality makes $D'=v_*D$
$\mathbb Q$-Cartier. The divisors $D$ and $v^*D'$ agree away
from the exceptional locus, hence agree everywhere because that
locus has no divisorial component. Any contracted curve would
then have $D$-degree zero, contrary to relative ampleness.
Thus $v$ is finite and birational with normal target, hence an
isomorphism.

We have $F\simeq B/K$. The quotient $\rho:B\to F$ is
quasi-\'etale by \eqref{smallhyp}, and purity makes it \'etale
because $F$ is smooth. Hirzebruch--Riemann--Roch now gives
\[
 1=\chi(B,\cO_B)=|K|\chi(F,\cO_F),
\]
where the first equality is Kodaira vanishing. Hence $|K|=1$;
$f=\rho^{-1}$ is the required isomorphism over $Z$.
\end{proof}

Return to $M=U_C(r,d)$, $N=SU_C(r,L)$, $J=\Pic^0(C)$, and
$H=J[r]$, with $g\geq3$, $r\geq2$, and $\gcd(r,d)=1$.
The next result requires no assumption on $\End(J)$ or $\Aut(C)$.

\begin{theorem}\label{classification-geometric-partner}
For every Fourier--Mukai partner $W$ of $M$, there are a
$g$-dimensional abelian variety $A$ and a group isomorphism
$\iota:H\xrightarrow{\sim}A[r]$ such that
\[
 W\simeq(N\times A)/H,\qquad
 \eta\cdot(E,a)=(E\otimes\eta,a+\iota(\eta)).
\]
\end{theorem}
\begin{proof}
Let $a:N\to Z=N/H$. Lemmas~\ref{classification-terminal} and
\ref{small} make $Z$ terminal and $\mathbb Q$-factorial and $a$
quasi-\'etale. Thus $K_N=a^*K_Z$ for compatible canonical
divisors, and $-K_Z$ is ample. Lemmas
\ref{canonical-global-generation} and \ref{antican} identify
the anticanonical contraction of $W$ with a morphism
$p_W:W\to Z$ having connected fibres.

We need the polarization as well as the anticanonical base.
Choose $m>1$ sufficiently divisible that $-mK_Z$ is very ample
Cartier, $\omega_W^{-m}$ is generated, and the corresponding
Veronese section rings are generated in degree one. The canonical
linearization on $\omega_N^{-m}$ is the pullback of that on
$\cO_Z(-mK_Z)$ over $Z_{\mathrm{reg}}$; normality extends this
identification across codimension two. Consequently
\[
 \bigoplus_{k\geq0}H^0(W,\omega_W^{-mk})
 \simeq\bigoplus_{k\geq0}H^0(N,\omega_N^{-mk})^H
 \simeq\bigoplus_{k\geq0}H^0(Z,\cO_Z(-mkK_Z)).
\]
The degree-one polarizations on Proj, pulled back to $W$, give
\begin{equation}\label{classification-canonical-polarization}
 \omega_W^{-m}\simeq p_W^*\cO_Z(-mK_Z).
\end{equation}

For a general smooth $D\in|-mK_W|$, the projective pair
$(W,D/m)$ is klt and $K_W+D/m\sim_{\mathbb Q}0$.
Ambro's theorem \cite[Theorem~0.1(2)]{Ambro} therefore makes the
Albanese morphism surjective with connected fibres and gives a
connected finite \'etale cover $\lambda:A'\to\Alb(W)$ and an
isomorphism
\[
 F\times A'\simeq W\times_{\Alb(W)}A'\qquad\text{over }A'.
\]
Here $F$ is projective and connected by the fibre assertion, and
smooth because the product is \'etale over the smooth variety
$W$. Choosing an origin over $0$ makes $A'$ an abelian variety
and $\lambda$ an isogeny. Denote the induced \'etale map by
$\pi':F\times A'\to W$. Derived invariance of irregularity
\cite[Corollary~B]{PopaSchnell} and Lemma~\ref{quotient-albanese}
give $\dim A'=g$.

Pulling back \eqref{classification-canonical-polarization} yields
\[
 (p_W\pi')^*\cO_Z(-mK_Z)
 \simeq\omega_{F\times A'}^{-m}
 \simeq\operatorname{pr}_F^*\omega_F^{-m}.
\]
Since $H^0(A',\cO_{A'})=\C$, every section of the right-hand
side comes from $F$. The morphism $p_W\pi'$, composed with the
embedding defined by $|-mK_Z|$, therefore factors through $F$.
Thus $p_W\pi'=q\operatorname{pr}_F$, where
$q(x)=p_W(\pi'(x,0))$ is surjective. Derived invariance of
dimension \cite[Proposition~4.1]{Huybrechts} gives
\[
 \dim F=\dim W-g=\dim M-g=\dim N=\dim Z,
\]
so $q$ is generically finite. Restriction to $F\times\{0\}$
gives $K_F\sim_{\mathbb Q}q^*K_Z$. By
Lemma~\ref{crepant-quotient}, we may identify $F=N$ so that
$q=a$ and $p_W\pi'=a\operatorname{pr}_N$.

Let $G=\ker\lambda$. The deck action on the Albanese pullback
has the form
\[
 \gamma_h(b,x)=(f_h(x)(b),x+h),\qquad f_h(x)\in\Aut(N).
\]
Since $\pi'\gamma_h=\pi'$ and $p_W\pi'=a\operatorname{pr}_N$,
each $f_h(x)$ lies in $\Aut_Z(N)=H$. To see that it is
independent of $x$ without imposing a condition on $\Aut(C)$,
choose $b_0\in a^{-1}(Z_{\mathrm{reg}})$. The morphism
$x\mapsto f_h(x)(b_0)$ from connected $A'$ to the finite fibre
$a^{-1}(a(b_0))$ is constant; freeness at $b_0$ then determines
$f_h(x)$. Write its constant value as $\rho(h)\in H$.
The action identities make $\rho:G\to H$ a homomorphism.

Put $G_0=\ker\rho$, $H_0=\rho(G)$, and $A=A'/G_0$.
Quotienting first by $G_0$ gives
\[
 W\simeq(N\times A)/H_0,\qquad
 \eta\cdot(b,x)=(\eta\cdot b,x+\iota(\eta)),
 \qquad \iota(\rho(h))=h+G_0.
\]
The last formula defines an embedding $H_0\hookrightarrow A$.
This presentation is over $Z$, with $p_W$ induced by
$a\operatorname{pr}_N$. For $z\in Z_{\mathrm{reg}}$,
\[
 p_W^{-1}(z)\simeq(a^{-1}(z)\times A)/H_0.
\]
The free transitive $H$-set $a^{-1}(z)$ has $[H:H_0]$
$H_0$-orbits. Each orbit contributes one disjoint component
isomorphic to $A$, since its $H_0$-action is free and transitive.
Connectedness of the fibres of $p_W$ therefore forces $H_0=H$.
Finally, $\dim A=g$ and $rH=0$, so $\iota(H)\subset A[r]$;
both groups have order $r^{2g}$. Hence $\iota$ identifies $H$
with all of $A[r]$.
\end{proof}

\begin{remark}\label{geometric-scope}
The geometric proof also applies to $g=2$ and $r\geq3$.
For $(g,r)=(2,2)$ the nontrivial tangent actions have age one,
so the argument excluding a crepant resolution does not apply.
Reconstruction for two given product quotients and the lifting
theorem below require only the codimension-two bound and apply
also in this case.
\end{remark}

\section{Equivalences of product quotients}\label{seclifting}\label{secconventions}
We first lift equivalences to the finite product covers and separate
their Fano and abelian factors. The final subsection determines when
an abelian equivalence kernel admits the linearization needed for
descent.

We retain the Hom, duality, and isometry conventions of
Section~\ref{secmodulisetup}. Write
$\Aut_{\mathrm{gp}}(A)=\End(A)^\times$ for automorphisms fixing
the origin; $\Aut(A)$ also includes translations.

Let $\mathcal P_A$ be the Poincar\'e line bundle on
$A\times\widehat A$, normalized to be trivial along both zero
sections, and put $P_\xi=\mathcal P_A|_{A\times\{\xi\}}$.
For $x\in A$, write $t_x(z)=z+x$ and $T_x=(t_x)_*$.
For $\xi\in\widehat A$, abbreviate $M_{P_\xi}$ to $M_\xi$.
For any line bundle $R$ on $A$, define
\[
 \phi_R:A\longrightarrow\widehat A,\qquad
 x\longmapsto[t_x^*R\otimes R^{-1}].
\]
We omit the subscript on $\mathcal P_A$ when its abelian variety
is clear. On $A\times A$, this notation requires a chosen principal
polarization identifying $A$ with $\widehat A$.

Recall the definition of an Orlov isometry in
\eqref{orlovisometries}. For an isometry from
$A_1\times\widehat A_1$ to $A_2\times\widehat A_2$, the blocks have
types $a:A_1\to A_2$, $b:\widehat A_1\to A_2$,
$c:A_1\to\widehat A_2$, and $e:\widehat A_1\to\widehat A_2$.

For a finite group $G$ acting on the left on $V$, a
\emph{$G$-linearization} of $K\in\Db(V)$ consists of isomorphisms
$\theta_g:g^*K\xrightarrow{\sim}K$ satisfying
\[
 \theta_e=\id_K,\qquad
 \theta_{gh}=\theta_h\circ h^*\theta_g,
 \qquad (gh)^*=h^*g^*.
\]
Here $e$ is the identity; the displayed composite has source
$h^*g^*K$. These are the inverses of the linearization maps in
\cite[Section~2.1]{KS}. For an abelian group $H$, written additively,
we equivalently use
$\theta_0=\id$ and
$\theta_{h+h'}=\theta_h\circ h^*\theta_{h'}$.
An \emph{invariant} object only satisfies $g^*K\simeq K$ for each
$g$; it need not admit isomorphisms obeying this cocycle condition.
Write $\Db_G(V)$ for the bounded derived category of equivariant
coherent sheaves, equivalently the category of linearized derived
objects in characteristic zero \cite[Section~2.1]{KS}.
For a subgroup $S\subset G$, let
\[
 \operatorname{Ind}_S^G:\Db_S(V)\longrightarrow\Db_G(V)
\]
denote induction, which is left adjoint
to restriction of linearizations from $G$ to $S$
\cite[Section~2.1]{KS}.
For representatives $T$ of the left cosets $G/S$, its underlying object is
\[
 \operatorname{Ind}_S^G(K)
 \simeq
 \bigoplus_{t\in T}(t^{-1})^*K
 \qquad\text{in }\Db(V),
\]
with the linearization induced by that of $K$ and the permutation
of cosets.

For a free $G$-action, finite \'etale descent gives an equivalence
\[
 \pi^*:\Db(V/G)\xrightarrow{\sim}\Db_G(V),\qquad
 \pi:V\longrightarrow V/G,
\]
with quasi-inverse $\Kcal\mapsto(\pi_*\Kcal)^G$
\cite[Section~4]{KS}; invariants are exact over $\C$.

\subsection{Equivariant lifting between product quotients}
Bridgeland--Maciocia treat canonical covers \cite[Theorem~4.5]{BM},
while Lombardi--Popa lift equivalences to cyclic covers associated with
torsion line bundles matched by the Rouquier isomorphism
\cite[Theorem~10]{LP}. Krug--Sosna give lift and descent criteria for
finite covers under kernel-linearization hypotheses
\cite[Propositions~4.2--4.3]{KS}. Here the quotient geometry produces,
for every equivalence, a lift with a kernel linearized under the graph
of a deck-group automorphism.

Let $B$ be a smooth connected projective Fano variety,
and let $H$ be a finite abelian group whose action
on $B$ satisfies \eqref{smallhyp}. 
For $i=1,2$, let $A_i$ be an abelian variety and fix
an injective homomorphism
$\iota_i:H\hookrightarrow A_i(\C)$.
Set
\[
 V_i:=B\times A_i,\qquad W_i:=V_i/H,
 \qquad
 h\cdot(b,x):=(h\cdot b,x+\iota_i(h)),
\]
and denote the finite \'etale quotient maps by
$\pi_i:V_i\to W_i$.

For $\nu\in\Aut(H)$, write
\( S_\nu:=\{(h,\nu(h)):h\in H\}\subset H\times H.\)
The action of $S_\nu$ on $V_1\times V_2$ is the
restriction of the product action of $H\times H$.

\begin{theorem}[Equivariant lifting]
\label{liftmain}\label{classification-two-cover-lift}
Every equivalence
$\Phi:\Db(W_1)\xrightarrow{\sim}\Db(W_2)$
admits an equivalence
$\widetilde\Phi:\Db(V_1)\xrightarrow{\sim}\Db(V_2)$
with an $S_\nu$-linearized Fourier--Mukai kernel
for some $\nu\in\Aut(H)$, satisfying
\begin{equation}\label{liftrelations}
 \widetilde\Phi\circ\pi_1^*
 \simeq \pi_2^*\circ\Phi,
 \qquad
 \pi_{2*}\circ\widetilde\Phi
 \simeq \Phi\circ\pi_{1*}.
\end{equation}
\end{theorem}

\begin{proof}
Let $P\in\Db(W_1\times W_2)$ be a Fourier--Mukai
kernel of $\Phi$, and let
$P'\in\Db(W_2\times W_1)$ be an inverse kernel.
Set
\( \Pi:=\pi_1\times\pi_2, 
 Q:=\Pi^*P,  G:=H\times H.\)
The pullback $Q$ carries its canonical $G$-linearization.
We first prove that
\begin{equation}\label{mainendo}
 \End(Q)\simeq\C^{|H|}
\end{equation}
as $\C$-algebras.

\medskip
\noindent
\emph{Restriction to the smooth quotient locus.}
Set $Z:=B/H$, and let $B^\circ\subset B$ be the locus
with trivial stabilizer.
By Lemma~\ref{small},
$B^\circ=a^{-1}(Z_{\mathrm{reg}})$, where
$a:B\to Z$ is the quotient map.
Let $p_i:W_i\to Z$ be the anticanonical contractions,
and set
\[
 U_i:=p_i^{-1}(Z_{\mathrm{reg}}),\qquad
 V_i^\circ:=\pi_i^{-1}(U_i)=B^\circ\times A_i.
\]
The complements of these open subsets have codimension
at least two.

By Lemma~\ref{antican}, there is an automorphism
$\sigma:Z\xrightarrow{\sim}Z$ such that
\( p_2(w_2)=\sigma(p_1(w_1))\) for every $(w_1,w_2)\in\Supp(P).$
Since $\sigma$ preserves $Z_{\mathrm{reg}}$, this gives
\[
 \Supp(P)\cap(U_1\times W_2)
 =
 \Supp(P)\cap(W_1\times U_2).
\]
Applying the same lemma to $\Phi^{-1}$ gives the
analogous equality for $P'$.

Write
\[
 P^\circ:=P|_{U_1\times U_2},\qquad
 (P')^\circ:=P'|_{U_2\times U_1},\qquad
 Q^\circ:=Q|_{V_1^\circ\times V_2^\circ}.
\]
These support equalities make both restricted kernels proper over
their source and target. They also force the intermediate coordinate
in either convolution to lie in the corresponding open subset.
Open base change \cite[Tag~08IB]{Stacks} therefore gives
\[
 \begin{aligned}
 (P')^\circ\star P^\circ
 &\simeq
 (P'\star P)|_{U_1\times U_1}
 \simeq\cO_{\Delta_{U_1}},\\
 P^\circ\star(P')^\circ
 &\simeq
 (P\star P')|_{U_2\times U_2}
 \simeq\cO_{\Delta_{U_2}}.
 \end{aligned}
\]
Thus $P^\circ$ and $(P')^\circ$ induce inverse equivalences.

\medskip
\noindent
\emph{Endomorphisms of the pulled-back kernel.}
For $i=1,2$, let $T_i\subset\Pic^0(W_i)$ be the
subgroup of character line bundles of $\pi_i$.
By \eqref{characters},
\[
 \pi_{i*}\cO_{V_i}
 \simeq\bigoplus_{\alpha\in T_i}\alpha,
 \qquad |T_i|=|H|.
\]
Let $q_i:W_1\times W_2\to W_i$ be the projections.
Since $\Pi$ is finite \'etale, adjunction, the
projection formula \cite[Tag~08EU]{Stacks}, and
the character decompositions give
\begin{equation}\label{lift-hom-decomposition}
 \begin{aligned}
 \End(Q)
 &\simeq\Hom(P,\Pi_*\Pi^*P)\\
 &\simeq
 \Hom\bigl(P,P\otimes\Pi_*\cO_{V_1\times V_2}\bigr)\\
 &\simeq
 \bigoplus_{\substack{\alpha\in T_1\\ \beta\in T_2}}
 \Hom\bigl(P,P\otimes q_1^*\alpha\otimes q_2^*\beta\bigr).
 \end{aligned}
\end{equation}
These are isomorphisms of $\C$-vector spaces.

Fix $\alpha\in T_1$ and $\beta\in T_2$, and set
\( K:=P'\star
 \bigl(P\otimes q_1^*\alpha\otimes q_2^*\beta\bigr).\)
Convolution with $P'$ identifies the corresponding
summand in \eqref{lift-hom-decomposition} with
\( \Hom(\cO_{\Delta_{W_1}},K),\)
and $K$ represents the autoequivalence
$\Phi^{-1}\circ M_\beta\circ\Phi\circ M_\alpha$.
By Rouquier's theorem
\cite[Theorem~4.18]{Rouquier}, there are
$t\in\Aut^0(W_1)$ and $R\in\Pic^0(W_1)$ such that
\( K\simeq\Gamma_{t*}R,\) where \( \Gamma_t:W_1\hookrightarrow W_1\times W_1, 
 x\longmapsto(x,t(x)).\)
Here $t$ and $R$ depend on $\alpha$ and $\beta$,
and no cohomological shift occurs.

For $t\ne\id$, a morphism from $\cO_{\Delta_{W_1}}$ to
$\Gamma_{t*}R$ has image supported on the proper intersection of
the two graphs. A line bundle on the integral graph has no such
subsheaf, so this Hom space vanishes. For $t=\id$, it is
$H^0(W_1,R)$, which is $\C$ if $R\simeq\cO_{W_1}$ and zero
otherwise, since $R\in\Pic^0(W_1)$.

It follows that every nonzero summand in
\eqref{lift-hom-decomposition} is one-dimensional
and satisfies
\[
 \Phi^{-1}\circ M_\beta\circ\Phi\circ M_\alpha
 \simeq\id_{\Db(W_1)}.
\]
For each $\alpha$ there is at most one such $\beta$, since
$M_\beta\simeq\Phi M_{\alpha^{-1}}\Phi^{-1}$ determines $\beta$
by evaluation on $\cO_{W_2}$. Hence
$\dim_\C\End(Q)\leq|H|$.

We next compare $\End(Q)$ with $\End(Q^\circ)$.
Every automorphism of $W_1$ induces an automorphism
of its anticanonical model and therefore preserves
$U_1$.
The adjunction and convolution identifications above
commute with restriction to the chosen open subsets.
Hence restriction on each summand of
\eqref{lift-hom-decomposition} corresponds to
\[
 \Hom(\cO_{\Delta_{W_1}},\Gamma_{t*}R)
 \longrightarrow
 \Hom\bigl(
   \cO_{\Delta_{U_1}},
   \Gamma_{t|_{U_1}*}(R|_{U_1})
 \bigr).
\]
If $t\ne\id_{W_1}$, then
$t|_{U_1}\ne\id_{U_1}$ by density, and both Hom
spaces vanish by the graph argument.
If $t=\id_{W_1}$, this is the restriction map
\( H^0(W_1,R)\longrightarrow H^0(U_1,R|_{U_1}).\)
For the inclusion $j:U_1\hookrightarrow W_1$,
normality, the codimension bound, and the local
freeness of $R$ give
\(R\xrightarrow{\sim}j_*(R|_{U_1}) 
  \text{\cite[Tag~0EBJ]{Stacks}}.\)
Thus restriction is an isomorphism in this case
as well.

Summing these isomorphisms gives
\[
 \operatorname{res}:\End(Q)\xrightarrow{\sim}\End(Q^\circ)
\]
as a $G$-equivariant algebra isomorphism: restriction preserves
composition and commutes with the pullbacks and linearization of
$Q$. Thus the required extension of morphisms follows from the
line-bundle calculation after convolution.

Lift $\sigma$ to an automorphism $f$ of $B$ by
Lemma~\ref{small}, and define
$\nu\in\Aut(H)$ by $\nu(h)=fhf^{-1}$.
The support of $Q^\circ$ is contained in the
pairwise disjoint closed subsets
\[
 D_h:=
 \bigl\{
 ((b,x),(b',y))\in V_1^\circ\times V_2^\circ:
 b'=h\cdot f(b)
 \bigr\},
 \qquad h\in H.
\]
Accordingly,
\[
 Q^\circ\simeq\bigoplus_{h\in H}Q_h^\circ,
 \qquad \Supp(Q_h^\circ)\subset D_h.
\]
The action on these subsets is
\( (h_1,h_2)\cdot D_h
 =D_{h+h_2-\nu(h_1)}.\)
Thus the canonical $G$-linearization permutes the
summands transitively.
Since $P^\circ$ induces an equivalence,
$Q^\circ\ne0$, so every $Q_h^\circ$ is nonzero.

The $|H|$ nonzero orthogonal projections onto these summands are
linearly independent. The dimension bound and restriction
isomorphism make them a basis of $\End(Q^\circ)$, with the
multiplication of $\C^{|H|}$. This proves \eqref{mainendo}
as an algebra statement.

\medskip
\noindent
\emph{An equivariant summand.}
The stabilizer of $D_0$ is
\( S:=\{(h,\nu(h)):h\in H\}\subset G.\)
Let $e_0^\circ$ be the projection onto $Q_0^\circ$,
and let $e_0\in\End(Q)$ be its preimage under
$\operatorname{res}$.
Since $\operatorname{res}$ is a $G$-equivariant
algebra isomorphism, $e_0$ is an $S$-invariant
primitive idempotent.

Restrict the $G$-linearization of $Q$ to $S$.
Then $e_0$ is an endomorphism of $Q$ in
$\Db_S(V_1\times V_2)$.
This category is idempotent complete, being the
bounded derived category of the abelian category
of $S$-equivariant coherent sheaves.
Splitting $e_0$ gives an $S$-equivariant decomposition
\[
 Q\simeq\widetilde P\oplus Q',
 \qquad e_0=\id_{\widetilde P}\oplus0.
\]
Moreover,
\[
 \End_{\Db(V_1\times V_2)}(\widetilde P)
 \simeq e_0\End(Q)e_0
 \simeq\C,
\]
so the underlying complex $\widetilde P$ is simple.

The translates of $e_0$, indexed by $G/S$, restrict to the
projections onto the $Q_h^\circ$. They are therefore orthogonal
and sum to $\id_Q$. Induction--restriction adjunction
\cite[Lemma~2.4]{KS} sends the $S$-equivariant inclusion
$\widetilde P\to Q$ to a $G$-equivariant isomorphism
\[
 \operatorname{Ind}_S^G\widetilde P\xrightarrow{\sim}Q,
\]
whose components are the corresponding summand inclusions.

\medskip
\noindent
\emph{Descent and the lifted equivalence.}
Set $\widetilde\Phi:=\Phi_{\widetilde P}$.
The $S$-linearization of $\widetilde P$ defines
an equivariant Fourier--Mukai functor
\[
 \Psi:\Db_H(V_1)\longrightarrow\Db_H(V_2)
\]
whose underlying functor is $\widetilde\Phi$
\cite[Section~3.2]{KS}.
We first identify $\Psi$ under finite \'etale descent.

Let $\widetilde q_i:V_1\times V_2\to V_i$ be the
projections, and put $H_1:=H\times\{0\}$.
For $E\in\Db(W_1)$, set
\[
 \mathcal T_E:=
 R\widetilde q_{2*}
 \bigl(Q\otimes\widetilde q_1^*\pi_1^*E\bigr),
 \qquad
 \mathcal F_E:=\widetilde\Phi(\pi_1^*E).
\]
Here $\pi_1^*E$ carries its canonical $H$-linearization.
Together with the $G$-linearization of $Q$, this
induces an $H_1$-action on $\mathcal T_E$, since
$H_1$ acts trivially on $V_2$.
Finite \'etale descent in the first factor and
flat base change in the second give a natural
isomorphism
\( \pi_2^*\Phi(E)\simeq\mathcal T_E^{H_1}.\)
The $H$-linearization on the right is induced by
the second factor $\{0\}\times H$.

The induced decomposition of $Q$ decomposes $\mathcal T_E$ into
summands indexed by $G/S$, with $\mathcal F_E$ at $S$. Since
$H_1\cap S=\{0\}$ and $H_1S=G$, the group $H_1$ permutes them
freely and transitively. If $\jmath_E$ and $p_E$ are the inclusion
and projection of $\mathcal F_E$, then
\[
 \psi_E:=
 p_E|_{\mathcal T_E^{H_1}}:
 \mathcal T_E^{H_1}\xrightarrow{\sim}\mathcal F_E
\]
is an isomorphism, with inverse induced by
\[
 \sum_{h\in H_1}\rho_h\circ\jmath_E:
 \mathcal F_E\longrightarrow\mathcal T_E,
\]
where $\rho_h$ is the action on $\mathcal T_E$. Both maps are
natural in $E$.

This isomorphism also preserves the target $H$-linearizations.
Indeed, for $k\in H$ and $h=\nu^{-1}(k)$, the identity
$(0,k)=(-h,0)+(h,k)$ identifies the action of $(0,k)$ on
$\mathcal T_E^{H_1}$ with that of $(h,k)\in S$. The latter
preserves $\mathcal F_E$ and acts by the linearizations of
$\widetilde P$ at $(h,k)$ and of $\pi_1^*E$ at $h$, exactly the
linearization defining $\Psi(\pi_1^*E)$.

Thus $\Psi\pi_1^*\simeq\pi_2^*\Phi$ in $\Db_H(V_2)$.
Finite \'etale descent makes $\Psi$ an equivalence; forgetting
linearizations gives the first identity in \eqref{liftrelations}.

It remains to prove that $\widetilde\Phi$ itself
is an equivalence.
Since $\widetilde P$ is simple,
\cite[Lemma~3.11]{KS} equips its right adjoint
kernel with a linearization for which the
adjunction unit and counit are equivariant.
These induce the unit and counit of the
equivariant adjunction for $\Psi$, and hence
are isomorphisms on all equivariant objects.

For any $K\in\Db(V_1)$, apply the equivariant
unit to $\operatorname{Ind}_{\{0\}}^H K$ and
forget equivariance.
The underlying object is
$\bigoplus_{h\in H}h^*K$, so the ordinary unit
is an isomorphism on this direct sum.
By naturality and additivity, its component
on the summand $K$ is the unit at $K$ and is
therefore an isomorphism.
Applying the same argument to the counit and
$\operatorname{Ind}_{\{0\}}^H L$, for
$L\in\Db(V_2)$, shows that the ordinary counit
is an isomorphism at every $L$.
Thus $\widetilde\Phi$ is an equivalence.

Finally, taking right adjoints of
$\widetilde\Phi\circ\pi_1^*
 \simeq\pi_2^*\circ\Phi$
gives
\( \pi_{1*}\circ\widetilde\Phi^{-1}
 \simeq\Phi^{-1}\circ\pi_{2*}.\)
Composing with $\Phi$ on the left and
$\widetilde\Phi$ on the right yields
\( \pi_{2*}\circ\widetilde\Phi
 \simeq\Phi\circ\pi_{1*}.\)
This proves both identities in \eqref{liftrelations}.
\end{proof}

\subsection{Equivalences of Fano--abelian products}\label{secproduct}
Sundelius's Theorem~4.1 concerns autoequivalences of $X\times E$,
where $E$ is elliptic and $X$ is smooth projective with
$\Pic^0(X)=0$ and $\Aut^0(X)=1$ \cite{Sundelius}.
Here the abelian factors may have arbitrary dimension and the equivalence
may join different products; the other factors are Fano, with no
discreteness assumption on their automorphism groups. Neither set of
geometric hypotheses contains the other. For the support argument,
compare also \cite[Appendix~A]{Sundelius}.

Let $B_i$ be a smooth connected projective Fano variety
and let $A_i$ be a complex abelian variety, for $i=1,2$.

\begin{theorem}[Product decomposition]\label{prodmain}
Every equivalence
\( \Phi:\Db(B_1\times A_1)
 \xrightarrow{\sim}\Db(B_2\times A_2)\)
admits a decomposition
\begin{equation}\label{abstractproduct}
 \Phi\simeq(f_*\circ M_R)\bo F,
\end{equation}
where $f:B_1\xrightarrow{\sim}B_2$ is an isomorphism,
$R\in\Pic(B_1)$, and
$F:\Db(A_1)\xrightarrow{\sim}\Db(A_2)$ is an equivalence.
\end{theorem}

No polarization or endomorphism-ring assumption is imposed on the
$A_i$. We first recover the Poincar\'e pairing from families of
standard kernels, then remove the abelian equivalence and recognize
point objects. Set $W_i=B_i\times A_i$.

\begin{lemma}
\label{geomproduct}
Every isomorphism $W_1\xrightarrow{\sim}W_2$ has the
form $f\times u$, where
$f:B_1\xrightarrow{\sim}B_2$ and
$u:A_1\xrightarrow{\sim}A_2$ are isomorphisms.
Moreover, for $i=1,2$, pullback along the projections induces
an isomorphism
\[
 \Pic(B_i)\oplus\Pic(A_i)
 \xrightarrow{\sim}\Pic(W_i),
 \qquad
 (L,M)\longmapsto L\bo M.
\]
\end{lemma}
\begin{proof}
Kodaira vanishing gives $\Pic^0(B_i)=\Alb(B_i)=0$, so projection
to $A_i$ is an Albanese morphism. An isomorphism
$\varphi:W_1\to W_2$ therefore has the form
$\varphi(b,a)=(f_a(b),u(a))$, with $u:A_1\to A_2$ an isomorphism
of varieties and $f_a:B_1\to B_2$ a family of isomorphisms.
Fixing $f_0$ defines a morphism
\[
 \gamma:A_1\longrightarrow\Aut(B_1),
 \qquad a\longmapsto f_0^{-1}\circ f_a.
\]

An anticanonical embedding realizes $\Aut(B_1)$ as the closed
stabilizer of $B_1$ in a projective linear group; in particular it
is affine. Since $A_1$ is connected and projective, $\gamma$ is
constant, and $\gamma(0)=\id$ gives $\varphi=f_0\times u$.

For $\mathcal L\in\Pic(W_i)$, put
$L_B=\mathcal L|_{B_i\times\{0\}}$. The restrictions to the
other $B_i$-fibres are algebraically equivalent to $L_B$, hence
isomorphic because $\Pic^0(B_i)=0$. The see-saw principle
\cite[Proposition~9.4]{Huybrechts} gives
$\mathcal L\simeq L_B\bo L_A$ for some $L_A\in\Pic(A_i)$.
Restriction to the two factors proves uniqueness.
\end{proof}
In particular,
\[
 \Aut(W_i)\simeq\Aut(B_i)\times\Aut(A_i),
 \qquad
 \Pic^0(W_i)\simeq\widehat A_i.
\]
A \emph{standard autoequivalence} of $W_i$ is a
functor $f_*\circ M_L[n]$, where
$f\in\Aut(W_i)$, $L\in\Pic(W_i)$, and $n\in\Z$.
The identity component of the group of their
isomorphism classes is
\[
 \Aut^0(W_i)\ltimes\Pic^0(W_i)
 \simeq G_i\times A_i\times\widehat A_i,
 \qquad G_i:=\Aut^0(B_i).
\]
Indeed, $\Pic^0(B_i)=0$ and translations of $A_i$
preserve algebraically trivial line-bundle classes.
The groups $G_i$ are affine by the preceding proof.
\begin{lemma}
\label{abelianRouquier}
For an equivalence $\Phi$ in Theorem~\ref{prodmain}, the Rouquier
isomorphism restricts to an isometric isomorphism
\[
 R_\Phi:A_1\times\widehat A_1\xrightarrow{\sim}
 A_2\times\widehat A_2.
\]
There is an equivalence $F:\Db(A_1)\xrightarrow{\sim}\Db(A_2)$
whose Rouquier isomorphism is $R_\Phi$.
\end{lemma}
\begin{proof}
Put $W_i=B_i\times A_i$ and $V_i=A_i\times\widehat A_i$.
Rouquier's theorem and Lemma~\ref{geomproduct} identify
$G_1\times V_1$ with $G_2\times V_2$, where $G_i=\Aut^0(B_i)$ is
affine. A homomorphism from an abelian variety to an affine group
is trivial. Applying this to the isomorphism and its inverse gives
an isomorphism $R=R_\Phi:V_1\to V_2$.
We must check that it is an isometry.

For $v=(x,\xi)$, let $\mathcal K_{i,v}$ be the standard kernel of
$\id_{\Db(B_i)}\bo(T_xM_\xi)$, where $T_x=(t_x)_*$.
Let $\mathcal K_i$ denote their universal family on
$V_i\times W_i\times W_i$. On the graph of
$(b,z)\mapsto(b,z+x)$, its line bundle is the pullback of
the normalized Poincar\'e bundle $\mathcal P_i$ on
$A_i\times\widehat A_i$ along $(z,\xi)$.
Let $\mathcal C_i$ be the line bundle on $V_i\times V_i$ obtained
by pulling back $\mathcal P_i$ along
$((x,\xi),(y,\eta))\mapsto(y,\xi)$. Thus its fibre at
$(v,w)$ is $(\mathcal P_i)_{(y,\xi)}$.
The translation identity for $\mathcal P_i$ gives an isomorphism
of relative convolution kernels
\begin{equation}\label{universalconvolution}
 \mathcal K_{i,v}\star\mathcal K_{i,w}
 \simeq\mathcal K_{i,v+w}\otimes\mathcal C_i(v,w).
\end{equation}

Let $P$ represent $\Phi$, let $P^R$ represent $\Phi^{-1}$, and let
$p:V_1\times W_2\times W_2\to V_1$ be projection. The family
$\mathcal Q=P\star\mathcal K_1\star P^R$ has the same fibres as
$\mathcal K'_2=(R\times\id_{W_2\times W_2})^*\mathcal K_2$, by the definition of
$R$ and uniqueness of equivalence kernels. There is therefore a
line bundle $D$ on $V_1$ with
\begin{equation}\label{universalconjugation}
 \mathcal Q\simeq\mathcal K'_2\otimes p^*D.
\end{equation}
To obtain this global isomorphism, put
\[
 \mathcal H=Rp_*R\mathcal Hom(\mathcal K'_2,\mathcal Q).
\]
The two kernel families are perfect on the smooth total space,
and $p$ is smooth and proper. Hence $\mathcal H$ is perfect and
derived base change identifies its fibre at $v\in V_1$ with
$R\Hom(\mathcal K'_{2,v},\mathcal Q_v)$.
The two fibre kernels are isomorphic, and their self-Ext groups
are those of $\cO_{\Delta_{W_2}}$. These fibres therefore have
zero negative cohomology, one-dimensional degree-zero cohomology,
and constant higher cohomology dimensions.
Cohomology and base change show that $\mathcal H$ has no negative
cohomology and that $D=\mathcal H^0(\mathcal H)$ is a line bundle
whose formation commutes with base change.
The truncation map $D\to\mathcal H$, followed by adjunction, gives
an evaluation map
\[
 \mathcal K'_2\otimes p^*D\longrightarrow\mathcal Q.
\]
On every fibre this is the evaluation from the one-dimensional
space of morphisms between two isomorphic simple kernels, hence
an isomorphism. Its cone is perfect and has zero derived fibre
at every point of $V_1$; derived Nakayama proves
\eqref{universalconjugation}.

Compare \eqref{universalconvolution} after conjugation, using the
adjunction $P^R\star P\simeq\cO_\Delta$.
Put $T=V_1\times V_1$, let $m:T\to V_1$ be addition, and let
$q:T\times W_2\times W_2\to T$ be the projection. For the family
$\mathcal S=(m\times\id_{W_2\times W_2})^*\mathcal K'_2$,
the resulting identity has the form
$\mathcal S\otimes q^*L_1\simeq\mathcal S\otimes q^*L_2$, where
\[
 \begin{aligned}
 L_1&=(R\times R)^*\mathcal C_2\otimes
       \operatorname{pr}_1^*D\otimes\operatorname{pr}_2^*D,\\
 L_2&=m^*D\otimes\mathcal C_1.
 \end{aligned}
\]
The same perfectness and base-change argument applies to
$Rq_*R\mathcal Hom(\mathcal S,\mathcal S)$.
Its degree-zero cohomology is a line bundle, and the identity
endomorphism of $\mathcal S$ gives an isomorphism
\[
 \cO_T\xrightarrow{\sim}
 \mathcal H^0\bigl(Rq_*R\mathcal Hom(\mathcal S,\mathcal S)\bigr),
\]
since it is nonzero on every fibre.
Applying $\mathcal H^0Rq_*R\mathcal Hom(\mathcal S,-)$ to the
identity between the two twisted families and using the projection
formula therefore gives $L_1\simeq L_2$, that is,
\[
 (R\times R)^*\mathcal C_2\otimes
 \operatorname{pr}_1^*D\otimes\operatorname{pr}_2^*D
 \simeq m^*D\otimes\mathcal C_1.
\]
If $\tau_i$ exchanges the factors of $V_i^2$, division by the
transposed identity cancels the symmetric $D$ terms and gives
\[
 (R\times R)^*(\mathcal C_2\otimes\tau_2^*\mathcal C_2^{-1})
 \simeq\mathcal C_1\otimes\tau_1^*\mathcal C_1^{-1}.
\]
The homomorphism associated with this alternating Poincar\'e
biextension is $j_i:V_i\to\widehat V_i$, $(x,\xi)\mapsto(\xi,-x)$.
Thus $\widehat Rj_2R=j_1$, which is exactly
\eqref{orlovisometries}. Orlov's realization theorem now supplies
$F$ \cite[Construction 4.10, Propositions 4.11--4.12 and Appendix]{Orlov}.
The factorization used for general isometries is proved in the
Appendix to the 2025 arXiv version.
\end{proof}

\begin{proof}[Proof of Theorem~\ref{prodmain}]
Choose $F_0:\Db(A_1)\xrightarrow{\sim}\Db(A_2)$ as in
Lemma~\ref{abelianRouquier}, and set
\[
 \Phi_0=(\id_{\Db(B_2)}\bo F_0^{-1})\Phi:
 \Db(B_1\times A_1)\xrightarrow{\sim}\Db(B_2\times A_1).
\]
It fixes all algebraically trivial tensor functors on the abelian
factor. The anticanonical contractions of these products are their
projections to $B_i$. By Lemma~\ref{antican}, for every closed point
$z\in B_1\times A_1$ the object $E_z=\Phi_0(k(z))$ is
supported in one fibre
$\{b\}\times A_1$ and satisfies $E_z\otimes p_{A_1}^*P_\xi\simeq E_z$
for every $\xi\in\widehat A_1$.

Let $\mathcal E=\mathcal H^i(E_z)\ne0$, and let $Z$ be a reduced
irreducible component of its support. Choose a smooth projective
resolution $j:\widetilde Z\to Z\hookrightarrow B_2\times A_1$,
and let $\mathcal V$ be the ordinary pullback $j^*\mathcal E$
modulo its torsion subsheaf. Nakayama's lemma at the generic point
of $Z$ gives $n=\operatorname{rk}\mathcal V>0$.
Tensoring by a line bundle preserves the torsion subsheaf, so the
invariances of $\mathcal E$ induce
$\mathcal V\otimes(p_{A_1}j)^*P_\xi\simeq\mathcal V$.
Taking the determinant of this torsion-free sheaf on the smooth
variety $\widetilde Z$ gives
\[
 (p_{A_1}j)^*(P_\xi^{\otimes n})\simeq\cO_{\widetilde Z}
 \qquad(\xi\in\widehat A_1).
\]
Surjectivity of $[n]$ on $\widehat A_1$ shows that
$(p_{A_1}j)^*:\Pic^0(A_1)\to\Pic^0(\widetilde Z)$ is zero.
Albanese--Picard duality then makes $p_{A_1}j$ constant. Since
$Z\subset\{b\}\times A_1$ and $j$ is surjective onto $Z$, the
component $Z$ is a point. Thus $E_z$ has zero-dimensional support.

Full faithfulness gives $\End(E_z)=\C$ and
$\Hom(E_z,E_z[k])=0$ for $k<0$. Together with zero-dimensional
support, the point-object criterion \cite[Lemma~4.5]{Huybrechts}
gives $E_z\simeq k(g(z))[n_z]$. The shift is locally constant
\cite[Corollary~6.14]{Huybrechts}, hence constant on the connected
source. The skyscraper criterion \cite[Corollary~5.23]{Huybrechts}
then gives
\[
 \Phi_0\simeq g_*M_Q[n],\qquad
 g:B_1\times A_1\xrightarrow{\sim}B_2\times A_1.
\]
Lemma~\ref{geomproduct} writes $g=f\times a$ and $Q=R\bo Q_A$.
Taking $F=F_0a_*M_{Q_A}[n]$ proves \eqref{abstractproduct}.
\end{proof}

In particular, lifting and product decomposition show that the abelian
factor $A$ of Theorem~\ref{classification-geometric-partner} satisfies
$\Db(A)\simeq\Db(J)$.

\subsection{Linearization under torsion translations}
With $T_x=(t_x)_*$ and $M_\xi=(-\otimes P_\xi)$ as above, the
isometry of an abelian autoequivalence is characterized by
\begin{equation}\label{grpOrlov}
 F(T_xM_\xi)F^{-1}\simeq T_{ax+b\xi}M_{cx+e\xi},
 \qquad \gamma(F)=\mat{a}{b}{c}{e}\in U(A\times\widehat A).
\end{equation}
Its kernel consists of the functors $T_xM_\xi[k]$
\cite[Proposition~3.3]{Orlov}.
When $A$ is principally polarized, fix a symmetric principal theta
line bundle $\Theta$ and put
\[
 \lambda=\phi_\Theta:A\xrightarrow{\sim}\widehat A,
 \qquad \phi_\Theta(x)=[t_x^*\Theta\otimes\Theta^{-1}].
\]
In principal-polarization coordinates, identify $A$ and
$\widehat A$ by $\lambda$, and write
$v^\dagger=\lambda^{-1}\widehat v\lambda$ for the Rosati involution
on $\End(A)$. In these coordinates an isometry has inverse
\[
 \mat{a}{b}{c}{e}^{-1}
 =\mat{e^\dagger}{-b^\dagger}{-c^\dagger}{a^\dagger}.
\]
All congruences between homomorphisms below mean divisibility in the
integral Hom group. A homomorphism $v:A\to A'$ vanishes on $A[n]$
if and only if it factors through $[n]:A\to A$, or equivalently
$v\in n\Hom(A,A')$.

With these conventions, in principal-polarization coordinates,
\begin{equation}\label{elementary-kernel-matrices}
 \gamma(M_R)=\mat{1}{0}{\lambda^{-1}\phi_R}{1},\qquad
 \gamma(\Phi_{\mathcal P^{-1}})=\mat{0}{1}{-1}{0},\qquad
 \gamma(\Phi_{d^*\Theta})=\mat{1}{1}{0}{1}.
\end{equation}
Here $d:A\times A\to A$ is the difference map, $d(x,y)=y-x$,
and $\mathcal P=(\id_A\times\lambda)^*\mathcal P_A$ is the
Poincar\'e bundle in principal-polarization coordinates.
The last equality follows from
$d^*\Theta\simeq p_1^*\Theta\otimes p_2^*\Theta\otimes\mathcal P^{-1}$.

\begin{lemma}[Translation linearizations between abelian varieties]
\label{general-translation-linearization}
Let $J$ be a complex principally polarized abelian variety with
principal polarization $\lambda:J\xrightarrow{\sim}\widehat J$,
let $A$ be a complex abelian variety, and let
$F:\Db(J)\xrightarrow{\sim}\Db(A)$ be an equivalence.
Write its Orlov isometry as
\[
 f=\begin{pmatrix}a&b\\c&e\end{pmatrix}:
 J\times\widehat J\xrightarrow{\sim}A\times\widehat A.
\]
Let $r\geq2$ and suppose that
$\beta:J[r]\xrightarrow{\sim}A[r]$ is an isomorphism.
The Fourier--Mukai kernel of $F$ admits a linearization under
\[
 \eta\cdot(x,y)=(x+\eta,y+\beta(\eta))
 \qquad(\eta\in J[r])
\]
if and only if
\[
 a|_{J[r]}=\beta,\qquad
 c\in r^2\Hom(J,\widehat A).
\]
\end{lemma}
\begin{proof}
Kernel invariance is equivalent to
$FT_\eta F^{-1}\simeq T_{\beta(\eta)}$, where
$T_x=(t_x)_*$.
It therefore requires and is implied by
$a|_{J[r]}=\beta$ and $c\in r\Hom(J,\widehat A)$.
Assume these conditions.
The isomorphism $J[r]\simeq A[r]$ gives $\dim J=\dim A$.
Since $a$ is injective on $J[r]$, the identity component of its
kernel has no nonzero $r$-torsion and hence has dimension zero.
Thus $a$ is an isogeny; its finite kernel has order relatively
prime to $r$.

\emph{A prime-to-$r$ kernel rank.}
We precompose $F$ with an upper shear to make $b$ an isogeny of
degree prime to $r$. This preserves $a$, $c$, the prescribed map
$\beta$, and the linearization condition in both directions.
The isometry identities imply that
$\widehat e b=\widehat b e$ is symmetric. Thus
\[
 s=\lambda^{-1}-\widehat e b:\widehat J\longrightarrow J
\]
is symmetric. Choose a line bundle $S$ on $\widehat J$ with $\phi_S=-s$;
such a line bundle exists because $s$ is symmetric.
For the normalized Poincar\'e transform
$\mathcal F=\Phi_{\mathcal P_J}:\Db(J)\to\Db(\widehat J)$,
the functor $C_s=\mathcal F^{-1}M_S\mathcal F$ induces the
upper shear
$u_s=\left(\begin{smallmatrix}1&s\\0&1\end{smallmatrix}\right)$.
Fourier transform exchanges convolution and tensor product, so
$C_s$ is convolution by $G_s=\mathcal F^{-1}(S)$.
Its kernel is $d^*G_s$, where $d(x,y)=y-x$, and therefore carries
the diagonal translation linearization pulled back from $G_s$.
The inverse $C_s^{-1}=\mathcal F^{-1}M_{S^{-1}}\mathcal F$ has
the same form. Convolution of equivariant kernels consequently
shows that $F$ is linearizable for the graph of $\beta$ if and
only if $FC_s$ is. Its isometry is
\[
 fu_s=\begin{pmatrix}a&as+b\\c&cs+e\end{pmatrix}.
\]
In particular, the first column is unchanged.
Since $a\widehat e-b\widehat c=1_A$, its upper-right block is
\[
 b'=as+b=a\lambda^{-1}-b\widehat c b.
\]
Modulo $r$, this agrees with $a\lambda^{-1}$.
Consequently $b'|_{\widehat J[r]}$ is an isomorphism, and $b'$
is an isogeny of degree relatively prime to $r$.
We may thus assume that $b$ has this property.

Orlov's construction
\cite[Construction 4.10 and Propositions 4.11--4.12]{Orlov} gives a
vector bundle $\mathcal E$ on $J\times A$ with
$\End(\mathcal E)=\C$, isometry $f$, rank
$\rho=\sqrt{\deg b}$, and
\[
 \frac{\phi_{\det\mathcal E}}{\rho}
 =
 \begin{pmatrix}
 b^{-1}a&-b^{-1}\\
 -(\widehat b)^{-1}&eb^{-1}
 \end{pmatrix}.
\]
The inverse $b^{-1}$ is taken in
$\Hom(A,\widehat J)\otimes_{\Z}\mathbb Q$.
The displayed homomorphism corresponds to the rational divisor
class $[\det\mathcal E]/\rho$ in $\NS(J\times A)\otimes_{\Z}\mathbb Q$.
Moreover, $\rho^2=\deg b$ is the degree of the projection of the graph
of $f$ to $J\times A$; in particular, $\gcd(\rho,r)=1$.
By \cite[Proposition~3.3]{Orlov}, any other equivalence with
isometry $f$ differs on the target by $T_xM_\xi[k]$.
Translations and shifts have diagonal $A[r]$-linearized kernels;
for $M_\xi$ this follows from the surjectivity of
$[r]^*:\Pic^0(A)\to\Pic^0(A)$ and descent through $[r]$.
Their inverses have the same property, so it suffices to test
$\mathcal E$.

\emph{The determinant obstruction.}
Let $\alpha\in H^2(J[r],\mathbb C^*)$ be the obstruction to
linearizing the simple invariant bundle $\mathcal E$
\cite[Lemma~1]{Ploog}; the coefficient action is trivial.
The graph map
$j_a:J\to J\times A$, $x\mapsto(x,a(x))$, is equivariant for
the specified action because $a|_{J[r]}=\beta$.
The line bundle $D=j_a^*\det\mathcal E$ has obstruction
$\rho\alpha$. To see this, write $g_\eta$ for the specified
translation on $J\times A$
and choose isomorphisms
$\theta_\eta:g_\eta^*\mathcal E\xrightarrow{\sim}\mathcal E$.
Simplicity gives a scalar two-cocycle $u(\eta,\zeta)$ satisfying
\[
 \theta_\eta\circ g_\eta^*\theta_\zeta
 =u(\eta,\zeta)\theta_{\eta+\zeta},
 \qquad [u]=\alpha.
\]
Taking determinants replaces $u(\eta,\zeta)$ by
$u(\eta,\zeta)^\rho$. Since $j_a$ is equivariant, pulling these
determinant isomorphisms back along $j_a$ leaves the same scalar
defect. Thus the obstruction of $D$ is $\rho\alpha$ in the same
group $H^2(J[r],\mathbb C^*)$.
For a finite abelian group $H$, divisibility of $\mathbb C^*$
gives $H^2(H,\mathbb C^*)\simeq
\Hom(\bigwedge^2H,\mathbb C^*)$.
Thus $H^2(J[r],\mathbb C^*)$ is killed by $r$.
Since $\gcd(\rho,r)=1$, multiplication by $\rho$ is an
automorphism of this group. The obstruction of $D$ therefore
vanishes if and only if $\alpha$ vanishes.

The isometry identities give
$c=eb^{-1}a-(\widehat b)^{-1}$.
Restricting the displayed slope homomorphism along $j_a$ gives
the particularly simple identity
\[
 \phi_D
 =\rho\bigl(-\widehat a(\widehat b)^{-1}
              +\widehat aeb^{-1}a\bigr)
 =\rho\,\widehat a c.
\]
\emph{Integral divisibility.}
A line bundle on $J$ is linearizable under translation by $J[r]$
if and only if it descends through $[r]:J\to J$.
Because $[r]^*$ is multiplication by $r^2$ on $\NS(J)$ and
is surjective on $\Pic^0(J)$, this is equivalent to
$[D]\in r^2\NS(J)$.
The Appell--Humbert description identifies $\NS(J)$ with the
subgroup of symmetric homomorphisms $J\to\widehat J$.
This subgroup is saturated: if $mv$ is symmetric for a nonzero
integer $m$ and $v\in\Hom(J,\widehat J)$, then
$m(v-\widehat v)=0$. The group $\Hom(J,\widehat J)$ is
torsion free, so $v=\widehat v$.
Hence this last condition is equivalent to
$\rho\,\widehat a c\in r^2\Hom(J,\widehat J)$.
Put $n=\deg(a)=\deg(\widehat a)$. Since $[n]_{\widehat A}$
kills $\ker\widehat a$, it factors through $\widehat a$,
giving an isogeny $z:\widehat J\to\widehat A$ with
$z\widehat a=[n]_{\widehat A}$.
Composing the divisibility condition with $z$ shows that
$\rho n c$ is divisible by $r^2$.
Since $\gcd(\rho n,r^2)=1$, there are integers $u,v$ with
$u\rho n+vr^2=1$. Hence
$c=u(\rho n c)+vr^2c$ is divisible by $r^2$.
The converse is immediate, proving the assertion.
\end{proof}

\begin{corollary}[The diagonal torsion criterion]\label{abeliandescent}
Let $A$ be a principally polarized abelian variety and
$F\in\Aut\Db(A)$ have Orlov matrix $\mat{a}{b}{c}{e}$.
Its kernel admits a diagonal $A[r]$-linearization if and only if
\begin{equation}\label{rankrcriterion}
 a-\id\in r\End(A),\quad e-\id\in r\End(\widehat A),\quad
 c\in r^2\Hom(A,\widehat A).
\end{equation}
Every integral isometry satisfying these conditions is realized
by such a linearized equivalence.
\end{corollary}
\begin{proof}
Apply Lemma~\ref{general-translation-linearization} with
source and target both $A$ and $\beta=\id$.
The isometry identity $\widehat a e-\widehat c b=\id$
makes $e\equiv\id\pmod r$ follow from
$a\equiv\id$ and $c\equiv0\pmod r$.
Orlov's realization theorem and the lemma give the last assertion.
\end{proof}

For tensoring by a line bundle $R$, the criterion is
$[R]\in r^2\NS(A)$, precisely descent through $[r]$.
Thus $M_{\Theta^r}$ has an invariant kernel with nonzero
linearization obstruction, whereas $[r]^*\Theta\simeq\Theta^{r^2}$
is linearizable.
\begin{lemma}[Linearizations of an exterior product]\label{factorlinearization}
Let a finite group $G$ act on smooth projective varieties $V$ and $W$.
Let $K_V$ be a coherent sheaf with $\End(K_V)=\C$, equipped
with a $G$-linearization, and let $K_W$ be a coherent sheaf up to
a shift. Then $K_V\bo K_W$ is $G$-linearizable for the diagonal
action if and only if $K_W$ is $G$-linearizable.
Fixing the linearization of $K_V$ identifies the two sets of
linearization structures.
\end{lemma}
\begin{proof}
Use the fixed isomorphism $g^*K_V\simeq K_V$, for $g\in G$.
Since negative Ext groups between coherent sheaves vanish,
the degree-zero K\"unneth formula gives
\[
 \Hom(K_V\bo g^*K_W,K_V\bo K_W)
 \simeq\End(K_V)\otimes\Hom(g^*K_W,K_W)
 \simeq\Hom(g^*K_W,K_W).
\]
These identifications respect composition and pullback.
The product isomorphisms therefore factor uniquely into the
fixed isomorphisms on $K_V$ and isomorphisms on $K_W$,
and the cocycle conditions agree.
\end{proof}

\begin{lemma}[Linearization of the theta generator]\label{thetalinearization}
For $N=SU_C(r,L)$ with $\gcd(r,d)=1$, the ample generator
$\Theta_N$ of $\Pic(N)$ admits a linearization for the action of
$H=J[r]$ by tensor product.
\end{lemma}
\begin{proof}
Every element of $H$ preserves $\Theta_N$, since it preserves the
ample generator of $\Pic(N)\simeq\Z$.
The pairs $(\eta,u)$ with $\eta\in H$ and
$u:s_\eta^*\Theta_N\xrightarrow{\sim}\Theta_N$, where
$s_\eta(E)=E\otimes\eta$, form its
theta group, a central extension
\[
 1\longrightarrow\C^*\longrightarrow\mathcal G_H(\Theta_N)
 \longrightarrow H\longrightarrow1.
\]
A splitting is precisely an $H$-linearization.
Let $e_r:H\times H\to\mu_r$ be the Weil pairing induced by
the principal polarization of $J$, where $\mu_r\subset\C^*$
is the group of $r$-th roots of unity.
The commutator of this extension is $e_r^{-r/\gcd(r,d)}$
\cite[Section~10, following Proposition~10.1]{BLS}.
Coprimality makes it trivial.
A commutative central extension of a finite abelian group by
$\C^*$ splits: lift cyclic generators and rescale them to have
the required orders.
\end{proof}
This linearization does not assert descent to the coarse quotient
$N/H$: such descent also requires stabilizers to act trivially on
the corresponding fibres of $\Theta_N$.

\section{Fourier--Mukai partners}\label{secpartners}
Keep $C,r,d,L,N,J,H$ as in the introduction.

\subsection{Proof of the classification}
Let $\mathcal P_r(J)$ be the set of isomorphism classes of pairs
$(A,\beta)$ satisfying \eqref{partnerisometrycriterion}; pairs are
identified by group isomorphisms $u:A\to A'$ with $u\beta=\beta'$.
We prove that $(A,\beta)\mapsto Y_{A,\beta}$ is a bijection
$\mathcal P_r(J)\simeq\operatorname{FM}(M)$ and that these sets are
finite. For isometry criteria for torsors under a fixed abelian variety
over a field, see \cite[Theorem~5.1]{AKW}.

\begin{proof}[Proof of Theorem~\ref{partner-pair-classification}]
For every pair in $\mathcal P_r(J)$, Orlov's existence theorem
and Lemma~\ref{general-translation-linearization} provide an
equivalence kernel from $J$ to $A$ with the indicated
$J[r]$-linearization. Replacing $\eta$ by $-\eta$ and taking
the exterior product with $\mathcal O_{\Delta_N}$ gives an
$H$-linearized equivalence kernel for the two free product actions,
with $H$ acting diagonally.
Equivariant Fourier--Mukai descent
\cite[Proposition 4.2]{KS} proves that $Y_{A,\beta}$
is a partner of $M$.

Conversely, Theorem~\ref{classification-geometric-partner} writes
every partner as $Y_{A,\beta}$, with $\beta=-\iota$ in that theorem.
Apply Theorems~\ref{liftmain} and~\ref{prodmain} to an equivalence
from $M$ to this quotient. The $N$-factor of the lift
is the pushforward by a tensoring automorphism of $N$,
followed by tensoring with
$\Theta_N^m$, with a shift absorbed in its abelian factor.
By \eqref{fixedaut}, its automorphism $f_N$ commutes with the
faithful $H$-action. Equivariance of the support gives
$f_Ns_\eta=s_{\nu(\eta)}f_N$, hence $\nu=\id_H$.
Lemma~\ref{thetalinearization} gives an
$H$-linearization on the kernel of the $N$-factor.
The abelian kernel is a coherent sheaf up to shift
\cite[Proposition 3.2]{Orlov}. Lemma~\ref{factorlinearization}
therefore applies and shows that this kernel inherits
a linearization under the graph subgroup of $J[r]\times A[r]$
defined by $\beta$, after replacing $\eta$ by $-\eta$.
Lemma~\ref{general-translation-linearization} then places
$(A,\beta)$ in $\mathcal P_r(J)$.

For the isomorphism criterion, Lemma~\ref{quotient-albanese}
identifies the Albanese morphisms as
$[E,x]\mapsto rx\in A$.
An isomorphism $Y_{A,\beta}\to Y_{A',\beta'}$ induces an
affine isomorphism between $A$ and $A'$.
Write this affine map as $x\mapsto u(x)+t$ and choose $s\in A'$
with $rs=t$. The map $x\mapsto u(x)+s$ lifts it through $[r]$;
the Cartesian Albanese squares therefore give an isomorphism
between the product covers.
Lemma~\ref{geomproduct} writes this lift as
$f_N\times(x\mapsto u(x)+s)$, where $u:A\to A'$ is a group
isomorphism.
The $N$-component is a tensor automorphism and therefore
centralizes the faithful $J[r]$ action.
The induced automorphism of the deck group is consequently
the identity. Comparison of the abelian components gives
$u\beta=\beta'$.
Conversely, this equality makes $\id_N\times u$ equivariant,
so it descends to an isomorphism of the quotients.

There are finitely many isomorphism classes of abelian
Fourier--Mukai partners of $J$ \cite[Corollary 2.20]{Orlov}, and, for each representative
$A$, there are only finitely many isomorphisms
$J[r]\to A[r]$. Thus $\mathcal P_r(J)$ is finite.
\end{proof}
\subsection{Explicit partners and the scalar case}
The scalar construction requires no restriction on the endomorphism
ring or on the action on the projective factor.

\begin{proposition}[Scalar quotient construction]\label{partnerconstruction}
Let $A$ be a principally polarized abelian variety, let $r\geq2$,
and let $B$ be a smooth connected projective variety with an action
of $H=A[r]$. For $k\in(\Z/r\Z)^\times$, let
\[
 Y_k=(B\times A)/H,\qquad
 \eta\cdot(b,x)=(\eta\cdot b,x-k\eta).
\]
Then $\Db(Y_1)\simeq\Db(Y_k)$.
\end{proposition}

\begin{proof}
Fix a symmetric principal theta line bundle $\Theta$, an integer
representative $k$, and $e\in\Z$ with $ke\equiv1\pmod{r^2}$.
On $A\times A$, using the
principal polarization to pull back the normalized Poincar\'e bundle, set
\begin{equation}\label{partnerkernel}
 K_k=p_1^*\Theta^k\otimes p_2^*\Theta^e\otimes\mathcal P^{-1}.
\end{equation}
Its transform is a composition of tensor functors and the
Poincar\'e equivalence. Its Orlov matrix is
$\mat{k}{1}{ke-1}{e}$.
Lemma~\ref{general-translation-linearization} therefore gives
a linearization for $(x,y)\mapsto(x+\eta,y+k\eta)$.
After replacing $\eta$ by $-\eta$, the exterior product with
$\cO_{\Delta_B}$ descends to the required equivalence
\cite[Proposition 4.2]{KS}.
\end{proof}

\begin{proof}[Proof of Corollary~\ref{complete-scalar-partners}]
For a pair $(A,\beta)$ in $\mathcal P_r(J)$, an Orlov isometry
identifies $A\times\widehat A$ with $J\times\widehat J\simeq J^2$.
The projection onto $A$ gives a rank-one idempotent in
$\End(J^2)=M_2(\mathbb Z)$. Its image and kernel are rank-one
direct summands of $\mathbb Z^2$, so a basis adapted to them
conjugates it to $\operatorname{diag}(1,0)$ over $\mathbb Z$.
Thus $A\simeq J$.
After this identification every block of the Orlov isometry
is an integer, so $\beta=[k]$ for some
$k\in(\mathbb Z/r\mathbb Z)^\times$.
Proposition~\ref{partnerconstruction} realizes every such $k$.
Finally, $\Aut_{\mathrm{gp}}(J)=\{1,-1\}$.
Its action on the scalar units is free for $r\geq3$;
for $r=2$ there is a single unit.
\end{proof}
For the scalar quotients, define
\begin{equation}\label{scalarsubgroup}
 S_r(J)=\{a\in(\Z/r\Z)^\times:
 u|_{J[r]}=[a]\text{ for some }u\in\Aut_{\mathrm{gp}}(J)\}.
\end{equation}
Here $\Aut_{\mathrm{gp}}(J)=\End(J)^\times$.

\begin{proposition}[Isomorphisms among the scalar quotients]\label{fm-main}
Let $C$ be a smooth connected complex projective curve of genus
$g\geq2$, let $r\geq2$, and assume $\gcd(r,d)=1$.
For every $k\in(\Z/r\Z)^\times$, there is an equivalence
\[
 \Db(U_C(r,d))\simeq\Db(X_k).
\]
If moreover $g\geq3$ and $\Aut(C)=1$, then
\begin{equation}\label{partnerisomorphism}
 X_k\simeq X_l\quad\Longleftrightarrow\quad
 lk^{-1}\in S_r(J).
\end{equation}
Consequently, under these additional hypotheses,
\begin{equation}\label{partnerexactcount}
 \#\bigl(\{X_k:k\in(\Z/r\Z)^\times\}/\simeq\bigr)
 =\frac{\varphi(r)}{|S_r(J)|},\qquad
 \#\operatorname{FM}(U_C(r,d))\geq\frac{\varphi(r)}{|S_r(J)|}.
\end{equation}
The scalar family includes $U_C(r,d)=X_1$.
\end{proposition}
\begin{proof}
Proposition~\ref{partnerconstruction} with $B=N$ and $A=J$
proves existence for every $g\geq2$.
For $g\geq3$ and $\Aut(C)=1$, Theorem~\ref{partner-pair-classification}
identifies $X_k$ and $X_l$ precisely when an automorphism
$u\in\Aut_{\mathrm{gp}}(J)$ satisfies $u[k]=[l]$ on $J[r]$.
This is the condition $lk^{-1}\in S_r(J)$.
The isomorphism classes are therefore the cosets of $S_r(J)$,
giving the count and the lower bound.
\end{proof}
\begin{corollary}[A lower bound from the torsion action]\label{partnerprimebound}
Assume $g\geq3$, $\Aut(C)=1$, and $\gcd(r,d)=1$.
Set
\[
 R_{2g}(r)=\{a\in(\Z/r\Z)^\times:a^{2g}=1\}.
\]
Then $S_r(J)\subseteq R_{2g}(r)$ and
\[
 \#\bigl(\{X_k\}/\simeq\bigr)
 \geq\frac{\varphi(r)}{|R_{2g}(r)|}.
\]
In particular, for every odd prime $p$ and $\gcd(p,d)=1$,
\begin{equation}\label{primepartnercount}
 \#\bigl(\{X_k:k\in(\Z/p\Z)^\times\}/\simeq\bigr)
 \geq\frac{p-1}{\gcd(2g,p-1)}.
\end{equation}
If $\gcd(g,(p-1)/2)=1$, then $S_p(J)=\{1,-1\}$, so
$X_k\simeq X_l$ if and only if $k\equiv\pm l\pmod p$,
and the family contains exactly $(p-1)/2$ isomorphism classes.
For a fixed curve $C$ as above, the numbers of constructed
pairwise nonisomorphic partners of $U_C(p,1)$ are unbounded
as $p$ ranges over the odd primes.
\end{corollary}
\begin{proof}
An automorphism $u\in\Aut_{\mathrm{gp}}(J)$ acts on
$\Lambda=H_1(J,\Z)$ with determinant $1$: its integral determinant
is $\pm1$ and is positive because its real extension is complex linear.
The natural isomorphism $\Lambda/r\Lambda\simeq J[r]$ shows
that $u|_{J[r]}=[a]$ implies that its matrix is scalar $a$
modulo $r$. Taking determinants gives $a^{2g}=1$ in $\Z/r\Z$.
Proposition~\ref{fm-main} gives the first bound. The cyclic group
$(\Z/p\Z)^\times$ has exactly
$\gcd(2g,p-1)$ roots of $z^{2g}=1$, proving
\eqref{primepartnercount}. If $\gcd(g,(p-1)/2)=1$, this
number is two. The inclusion $\{1,-1\}\subseteq S_p(J)$
then gives the asserted equality. Finally, the right-hand side
of \eqref{primepartnercount} is at least $(p-1)/(2g)$,
which tends to infinity.
\end{proof}
For example, take $r=5$, a degree $d$ coprime to $5$, and a curve
of odd genus $g\geq3$ with $\Aut(C)=1$.
Then $S_5(J)=\{\pm1\}$ without any condition on $\End(J)$.
The choice $k=2$, $e=13$ in \eqref{partnerkernel} gives a partner
$X_2\not\simeq M$.

\begin{corollary}[Nonisomorphic partners are not moduli spaces of vector bundles]\label{nonmoduli-partners}
Assume $g(C)\geq4$ and $\Aut(C)=1$.
For $k\in(\Z/r\Z)^\times\setminus S_r(J)$,
the variety $X_k$ of Proposition~\ref{fm-main} is not isomorphic
to any $U_{C'}(r',d')$ with $C'$ a smooth connected projective curve
of genus at least two, $r'\geq2$, $d'\in\Z$, and $\gcd(r',d')=1$.
\end{corollary}
\begin{proof}
If $X_k\simeq U_{C'}(r',d')$, then the two bundle moduli spaces
are derived equivalent. Corollary~\ref{moduli-torelli} gives
$C'\simeq C$ and $r'=r$; in particular $g(C')\geq4$.
Theorem~\ref{moduli-derived-classification} then identifies
$U_{C'}(r',d')$ with $U_C(r,d)=X_1$, contrary to
Proposition~\ref{fm-main} and $k\notin S_r(J)$.
\end{proof}

\subsection{Examples with real multiplication}
Extra endomorphisms can identify scalar quotients and can also give
partners not represented by those quotients. We exhibit both effects
on a single complex Jacobian.

\begin{example}[Two effects of real multiplication]
\label{real-multiplication-examples}\label{ranksevenRM}
\label{rank-nineteen-nonscalar-partner}
Let $\zeta_7$ be a primitive seventh root of unity.
There is a smooth nonhyperelliptic complex curve $C$ of genus three
with $\Aut(C)=1$ and $\End(J(C))=\mathcal O_F$, where
\[
 F=\mathbb Q(t),\qquad t=\zeta_7+\zeta_7^{-1},\qquad
 t^3+t^2-2t-1=0,\qquad\mathcal O_F=\Z[t],
\]
for which the following hold.
In rank seven and every degree coprime to seven, all six scalar
quotients $X_k$ are isomorphic.
In rank nineteen and degree one, there is a partner with abelian
factor $J$ which is not isomorphic to any scalar quotient.
\end{example}
\begin{proof}
We first justify the existence of $C$ with the specified endomorphism
ring. The construction in \cite[Theorem 3.7]{HLL} gives a family of
nonhyperelliptic genus-three curves with an integral
$\mathcal O_F$-action. This action is self-adjoint for the canonical
principal polarization. Indeed, in the dihedral construction let
$q:X\to C=X/\langle\tau\rangle$ be the double quotient and let
$\sigma$ have order seven, with $\tau\sigma\tau=\sigma^{-1}$.
The endomorphism representing $t$ is $q_*\sigma_*q^*$; its Rosati
adjoint is $q_*\sigma^{-1}_*q^*$, which is the same endomorphism
because $q\tau=q$.

Section~5 of \cite{HLL} computes rank three for the deformation
map to the moduli of curves. At a smooth plane quartic, the dual
of the infinitesimal Torelli map is
$\operatorname{Sym}^2H^0(C,\omega_C)\to H^0(C,\omega_C^2)$;
it is an isomorphism since the canonical plane quartic lies on no
quadric. Thus the associated period map also has rank three.
On a simply connected neighbourhood, mark the integral homology,
the polarization, and the endomorphism $t$. The period map then
takes values in the three-dimensional domain of polarized
$\mathcal O_F$-linear Hodge structures. The holomorphic submersion
theorem shows that its image contains a nonempty analytic open
subset of this domain, or equivalently of the corresponding
Hilbert modular component. This analytic openness is sufficient
for the following argument.

Put $V=H_1(J,\mathbb Q)\simeq F^2$ on this marked open subset.
The period domain is a product of three upper half-planes:
the complex structure varies independently on each summand of
$V\otimes\mathbb R=\bigoplus_{\sigma:F\hookrightarrow\mathbb R}\mathbb R^2$.
An operator commuting with every such complex structure has no
off-diagonal blocks and is scalar on each real summand. If the
operator is rational, these scalars are the three embeddings of
one element of $F$. Thus the common rational centralizer is exactly
$F$. For each $q\in\End_{\mathbb Q}(V)\setminus F$, the condition
that $q$ preserve the Hodge structure defines a proper closed
analytic subset of the chart. There are countably many such $q$.
Baire's theorem gives a point outside their union, at which
$\End(J)\otimes\mathbb Q=F$. Since $\End(J)$ is an order containing
the maximal order $\mathcal O_F$, it equals $\mathcal O_F$.

Every curve automorphism has finite order, so its action on $J$
is $1$ or $-1$, the only roots of unity in $F$. Its projective action
on $H^0(C,\omega_C)$ is consequently trivial. The canonical map is
an embedding for this nonhyperelliptic curve, so $\Aut(C)=1$.

Fix a curve $C$ obtained in this way, and put $J=J(C)$.
For rank seven, $t$ is a unit with inverse $t^2+t-2$, and
\[
 t^7=9+14t-14t^2\equiv2\pmod{7\mathcal O_F}.
\]
Thus $2,-1\in S_7(J)$; these generate $(\Z/7\Z)^\times$.
Proposition~\ref{fm-main} gives the first assertion, and
$\id_N\times t^7$ explicitly induces $X_1\simeq X_2$.

For rank nineteen, put $v=t-3$ and $e=112(t^2+4t+10)$.
The defining polynomial gives
\[
 (t-3)(t^2+4t+10)=-29,\qquad ve=1-9\cdot19^2.
\]
Since $F$ is Rosati-fixed, the identification of $\NS(J)$ with
the Rosati-symmetric integral endomorphisms gives line bundles $L_v,L_e$ with
$\lambda^{-1}\phi_{L_v}=v$ and $\lambda^{-1}\phi_{L_e}=e$,
where $\lambda$ is the principal polarization of $J$. By
\eqref{elementary-kernel-matrices}, the line bundle kernel
\[
 K_v=p_1^*L_v\otimes p_2^*L_e\otimes\mathcal P^{-1}
\]
defines an equivalence with matrix
\[
 \mat{t-3}{1}{-9\cdot19^2}{112(t^2+4t+10)}.
\]
Since $v$ is invertible modulo $19$,
Lemma~\ref{general-translation-linearization} gives the required
linearization. Hence, for $N=SU_C(19,L)$ with $\deg L=1$,
\[
 Y_v=(N\times J)/J[19],\qquad
 \eta\cdot(E,x)=(E\otimes\eta,x-v\eta)
\]
is a partner of $U_C(19,1)$.

If $Y_v\simeq X_k$, the isomorphism classification would give
$u\in\mathcal O_F^\times$ with $uv\equiv k\pmod{19\mathcal O_F}$.
Norms would then give $\pm9\equiv k^3\pmod{19}$, since
$N_{F/\mathbb Q}(v)=-29$. The nonzero cubes modulo $19$ are
$\{1,7,8,11,12,18\}$, containing neither $9$ nor $-9$.
This proves the second assertion.
\end{proof}

\section{Autoequivalence groups}\label{secmoduli}
\subsection{The autoequivalence exact sequence}\label{secgroupsequence}
Return to the tensor-product cover $\pi:N\times J\to M$.

Theorems~\ref{liftmain} and \ref{prodmain}, together with
$\Pic(N)=\Z[\Theta_N]$, give every autoequivalence a lift
\begin{equation}\label{productform}
 (f_*M_{\Theta_N^m})\bo F,\qquad
 f\in\Aut(N),\quad m\in\Z,\quad F\in\Aut\Db(J).
\end{equation}
This assertion holds in every genus $g\geq2$.

\begin{lemma}[Uniqueness of lifts]\label{coverliftuniqueness}
Let $\pi:Y\to X$ be a connected finite Galois cover of smooth projective
varieties. Two underlying autoequivalences lifting the same autoequivalence through
$\pi$, with the pullback and pushforward
isomorphisms of Theorem~\ref{liftmain}, differ by a deck transformation.
\end{lemma}
\begin{proof}
Their quotient $\mathcal L$ lifts the identity. For each closed point $y$,
$\pi_*\mathcal L(k(y))\simeq k(\pi(y))$.
Finite pushforward is exact and faithful, so $\mathcal L(k(y))$
is an unshifted length-one skyscraper in the fibre over $\pi(y)$.
The skyscraper criterion gives $\mathcal L\simeq f_*M_T$
\cite[Corollary~5.23]{Huybrechts}.
Now $\pi f=\pi$ makes $f$ a deck transformation, while
$\mathcal L\pi^*\cO_X\simeq\pi^*\cO_X$ gives $T\simeq\cO_Y$.
\end{proof}

For the remainder of this section assume
\begin{equation}\label{modgeneric}
 g\geq3,\qquad \Aut(C)=1.
\end{equation}
Equation~\eqref{fixedaut} identifies $\Aut(N)$ with the tensor
actions of $H$.
If a lift has geometric factor $f$ and deck-group automorphism $\nu$,
then $f s_\eta f^{-1}=s_{\nu(\eta)}$.
Since $H$ is abelian and acts faithfully, $\nu=\id_H$.
Removing a deck transformation absorbs $f$ into a translation on $J$.
Thus every autoequivalence has a lift of the form
\begin{equation}\label{modreducedlift}
 M_{\Theta_N^m}\bo F,
 \qquad m\in\Z,\quad F\in\Aut\Db(J),
\end{equation}
with a diagonal $H$-linearization.
By Lemma~\ref{coverliftuniqueness}, two such lifts differ by a deck
transformation. Their kernels have diagonal support in the $N$ factors,
so faithfulness forces that transformation to be the identity.
Their underlying kernels are therefore isomorphic. Abelian equivalence
kernels are sheaves up to shift \cite[Proposition 3.2]{Orlov},
so this isomorphism identifies the
shifts and gives $\Theta_N^m\bo E\simeq\Theta_N^{m'}\bo E'$.
Choose a closed point $z\in J^2$ with $E\otimes k(z)\neq0$.
Ordinary restriction to $N\times\{z\}$ gives an isomorphism
$(\Theta_N^m)^{\oplus n}\simeq(\Theta_N^{m'})^{\oplus n}$ for some
$n>0$: equality of the ranks first identifies the two fibre dimensions.
Taking determinants gives $n(m-m')[\Theta_N]=0$, hence $m=m'$
because $\Pic(N)$ is torsion free. Restriction to $\{p\}\times J^2$
for any closed point $p\in N$ then identifies $E$ and $E'$.
Hence the normalized underlying lift is unique; composition adds
$m$ and multiplies its abelian Orlov matrices. This uniqueness
does not fix a linearization: the linearizations of its simple
kernel form a torsor under $\Hom(H,\C^*)$. Their descents can
differ by the character line bundles of the cover.

\begin{lemma}[The central theta line bundle]\label{thetadescent}
The line bundle $\Theta_N$ admits an $H$-linearization. Fix one,
and let $Q_N$ be the descent of $\Theta_N\bo\cO_J$ to $M$.
For the tensoring action
$\mathfrak a:M\times J\to M$, $(E,x)\mapsto E\otimes x$,
$Q_N$ has a $J$-linearization
\begin{equation}\label{grpQaction}
 \mathfrak a^*Q_N\simeq\operatorname{pr}_M^*Q_N.
\end{equation}
Moreover, $M_{Q_N}$ is central in $\Aut\Db(M)$.
\end{lemma}
\begin{proof}
Use the linearization supplied by
Lemma~\ref{thetalinearization} and descend
$\Theta_N\bo\cO_J$ along the free product action.
Translations on the second factor preserve the linearized line
bundle and commute with $H$, giving \eqref{grpQaction} with its
unit and multiplication compatibilities.

A lift of the form~\eqref{modreducedlift} is supported on $\Delta_N\times J^2$.
On this subvariety the input and output pullbacks of $\Theta_N$
coincide, including their linearizations. Tensor associativity
therefore gives an $H$-equivariant isomorphism between the kernels
of the lift composed on either side with $M_{\pi^*Q_N}$.
Descending proves centrality.
\end{proof}

Fix a symmetric principal theta line bundle $\Theta$ on $J$.
To apply Corollary~\ref{abeliandescent}, put
$A_m=\Delta_{N*}\Theta_N^m$ with its fixed linearization, and let
$K_F$ be a kernel of $F$ in \eqref{modreducedlift}.
Lemma~\ref{factorlinearization} identifies
linearizations of $A_m\bo K_F$ with those of $K_F$.
Replacing $\eta$ by $-\eta$ gives precisely the action in
Corollary~\ref{abeliandescent}.

We define the maps in \eqref{grpExactSequence} and prove exactness.
Recall the congruence subgroup $U(J\times\widehat J;r)$ from
\eqref{generalcongruence}.
For $x\in J$, let $\mathfrak a_x(E)=E\otimes x$ be the tensoring
automorphism of $M$. Write $\delta(E)=\det(E)\otimes L^{-1}$ and set
\begin{equation}\label{grpStandardM}
 S_{x,\xi}=(\mathfrak a_x)_*M_{\delta^*P_\xi},
 \qquad x\in J(\C),\quad\xi\in\widehat J(\C).
\end{equation}
For $m,k\in\Z$, define
\[
 \iota(m,k,x,\xi)=M_{Q_N^m}[k]S_{x,\xi}.
\]
Here $Q_N$ is fixed as in Lemma~\ref{thetadescent}, and $[k]$
is the cohomological shift. The other map, $R$, is
characterized by conjugation on the standard functors:
\[
 \Phi S_{x,\xi}\Phi^{-1}\simeq S_{R(\Phi)(x,\xi)}.
\]
The proof below shows that this defines a homomorphism into
$U(J\times\widehat J;r)$.

\begin{proof}[Proof of the exact sequence]
The functors $S_{x,\xi}$ identify $J\times\widehat J$ with the full Rouquier
group $\Aut^0(M)\times\Pic^0(M)$. Indeed,
$\Pic^0(M)\simeq\widehat J$ by Lemma~\ref{quotient-albanese}.
The map $J\to\Aut^0(M)$ is faithful: $\mathfrak a_x=\id$ first
gives $rx=0$ on determinants, then $x=0$ by faithfulness on $N$.
Moreover, since $\pi$ is \'etale and $\Aut(N)$ is finite,
\[
 H^0(M,T_M)=H^0(N\times J,T_{N\times J})^H\simeq\C^g.
\]
Thus the closed embedding $J\to\Aut^0(M)$ is an isomorphism.
Rouquier's theorem \cite[Theorem 4.18]{Rouquier} consequently defines
a homomorphism $R$ by actual conjugation on the $S_{x,\xi}$.

The equality $\delta\pi=[r]\operatorname{pr}_J$ shows that the
normalized lift of $S_{x,\xi}$ has abelian parameters $(x,r\xi)$:
the translation parameter is unchanged, whereas
$[r]^*P_\xi\simeq P_{r\xi}$. If the normalized lift of $\Phi$ has abelian matrix
$\gamma(F)=\mat{a}{b}{c}{e}$, uniqueness of normalized lifts gives
\[
 D R(\Phi)=\gamma(F)D,\qquad
 D=\operatorname{diag}(\id_J,r\id_{\widehat J}).
\]
Any two such homomorphisms differ by a homomorphism from the connected
variety $J\times\widehat J$ to the finite group $\ker D$, hence agree.
Explicitly,
\begin{equation}\label{Rmatrix}
 R(\Phi)=\mat{a}{rb}{c/r}{e},\qquad
 \Phi S_{x,\xi}\Phi^{-1}
 \simeq S_{ax+rb\xi,(c/r)x+e\xi}.
\end{equation}
The matrix $D$ multiplies the alternating Poincar\'e pairing by
$r$, so conjugation by $D$ preserves rational Orlov isometries.
Corollary~\ref{abeliandescent} gives $a,e\equiv\id\pmod r$ and
$c\in r^2\Hom(J,\widehat J)$. Thus every block in
\eqref{Rmatrix} is integral and both off-diagonal blocks are
divisible by $r$; hence $R(\Phi)\in U(J\times\widehat J;r)$.
Conversely, for
$v=\mat{a}{B}{C}{e}$ in this group,
\[
 DvD^{-1}=\mat{a}{B/r}{rC}{e}
\]
is an integral isometry satisfying \eqref{rankrcriterion}.
Indeed, the level-$r$ condition makes $B$ and $C$ vanish on
the corresponding $r$-torsion groups. Factoring through $[r]$
gives $B\in r\Hom(\widehat J,J)$ and
$C\in r\Hom(J,\widehat J)$ in the integral Hom groups.
Consequently $B/r$ is integral and $rC$ is divisible by $r^2$;
the isometry inverse, formed from the dual blocks, is integral as well.
Corollary~\ref{abeliandescent} realizes it by a linearized equivalence;
its product with $\cO_{\Delta_N}$ descends. Thus $R$ is surjective.

For the map $\iota$ defined above, translations preserve algebraically trivial line-bundle classes,
and $M_{Q_N}$ and the shift are central. Hence $\iota$ is a
homomorphism into $\ker R$.
If its value is the identity, uniqueness of normalized lifts gives
$m=k=0$, $x=0$, and $r\xi=0$. The original functor is then
$M_{\delta^*P_\xi}$, and injectivity of $\delta^*$ forces $\xi=0$.
So $\iota$ is injective.

If $R(\Phi)=I$, Orlov's kernel description gives
$F\simeq T_xM_y[k]$. Choose $\xi$ with $r\xi=y$.
Then $\Phi$ and $\iota(m,k,x,\xi)$ have the same underlying lift.
Linearizations of a simple kernel differ by a character of $H$
\cite[Lemma 1]{Ploog}. Here the character group $\Hom(H,\C^*)$ is identified with
$\widehat J[r]$ by the perfect pairing between $J[r]$ and
$\widehat J[r]$. The corresponding descent line bundles on $M$
are exactly $\delta^*P_\beta$, $\beta\in\widehat J[r]$, because
the cover is the pullback of $[r]:J\to J$ along $\delta$.
Replacing $\xi$ by $\xi+\beta$ gives $\Phi$, proving exactness.
The conjugation assertion is built into the definition of $R$.
Finally, the exponent $m$ is a homomorphism sending the central
$M_{Q_N}$ to $1$, so
\[
 \Aut\Db(M)=\langle M_{Q_N}\rangle\times\ker(\Phi\mapsto m).
\]
\end{proof}

\subsection{Splitting in the scalar case}
\begin{proof}[Proof of the splitting in Theorem~\ref{fullgroup}]
If $\End(J)=\Z$, the principal polarization identifies
$U(J\times\widehat J;r)$ with $\Gamma(r)$.
It suffices to split \eqref{grpExactSequence}.

For $r\geq3$, $\Gamma(r)$ is free
\cite[Remark 3.10]{ConradSL2}; arbitrary lifts of a free basis
therefore give a section of \eqref{grpExactSequence}.

For $r=2$, let $C_\Theta$ be convolution by $\Theta$.
Its kernel $d^*\Theta$, $d(x,y)=y-x$, has the canonical diagonal
$H$-linearization. Its product with $\cO_{\Delta_N}$ descends to
an equivalence $\mathsf U$ by \cite[Proposition~4.2]{KS}.
The identification
$(\Delta_N\times J^2)/H\simeq M\times J$,
$(E,x,y)\mapsto(E\otimes x,y-x)$, gives
\[
 \mathsf U(\mathcal E)=R\mathfrak a_*(\mathcal E\bo\Theta).
\]
Set
\begin{equation}\label{generators}
 \mathsf V=M_{\delta^*\Theta},\qquad
 \mathsf K=\kappa_*,\qquad \kappa(E)=E^\vee\otimes L.
\end{equation}
Their Jacobian lifts are $C_\Theta$, $M_{[2]^*\Theta}$, and
$[-1]_*$; the last uses $E^\vee\otimes L\simeq E$ on $N$.
Consequently
\[
 R(\mathsf U)=u=\mat{1}{2}{0}{1},\qquad
 R(\mathsf V)=v=\mat{1}{0}{2}{1},\qquad R(\mathsf K)=-I.
\]
These matrices $u,v,-I$ generate $\Gamma(2)$
\cite[Theorem 3.1]{ConradSL2}. Their projective actions satisfy,
for every $n\neq0$,
\[
 u^n\{z:|z|<1\}\subset\{z:|z|>1\}\cup\{\infty\},
 \qquad
 v^n\bigl(\{z:|z|>1\}\cup\{\infty\}\bigr)\subset\{z:|z|<1\}.
\]
Ping-pong therefore makes $u,v$ free generators of their
projective image, so $-I\notin\langle u,v\rangle$ and
$\Gamma(2)=\langle u,v\rangle\times\langle-I\rangle
\simeq F_2\times\Z/2\Z$, where $F_2$ denotes the free group on two
generators.
The identities $\kappa^2=\id$, $\delta\kappa=-\delta$, and
$\kappa\mathfrak a=\mathfrak a(\kappa\times[-1])$, together with
$[-1]^*\Theta\simeq\Theta$, give
\[
 \mathsf K^2\simeq\id,\qquad
 \mathsf K\mathsf U\simeq\mathsf U\mathsf K,
 \qquad \mathsf K\mathsf V\simeq\mathsf V\mathsf K.
\]
Thus $(u,v,-I)\mapsto(\mathsf U,\mathsf V,\mathsf K)$ defines a section.
In both cases the kernel's $\Z^2$ factor is central, and
\eqref{Rmatrix} gives the remaining action, proving \eqref{fullformula}.
\end{proof}

\begin{remark}[Scope of the group formula]\label{genustwoscope}
The lifting and product-decomposition statements apply also in genus two.
The group calculation assumes $\Aut(C)=1$ and therefore excludes
that genus. For noncoprime rank and degree the projective semistable
moduli space is generally singular and its stable locus generally
nonprojective, so the smooth-projective arguments do not apply.
\end{remark}

\begin{remark}[Additional N\'eron--Severi classes]
Let $T$ be a line bundle on $J$ whose class in $\NS(J)$ is not
an integral multiple of $[\Theta]$. Then $M_{\delta^*T}$ has lift
$\id_{\Db(N)}\bo M_{[r]^*T}$ and
\[
 R(M_{\delta^*T})=\mat{\id_J}{0}{r\phi_T}{\id_{\widehat J}}.
\]
This belongs to $U(J\times\widehat J;r)$ but not to its scalar
subgroup $\Gamma(r)$. The free-group argument above does not
establish a splitting for the larger group; determining such a
splitting requires computing the remaining extension class.
\end{remark}

\section*{Acknowledgement}
This paper is based on the author’s master’s thesis. The author thanks
Yinbang Lin for bringing this problem to his attention.
\section*{AI-assisted preparation}
Generative artificial-intelligence tools were used for language polishing.
Under detailed author direction. The authors supplied the core ideas, precise
mathematical constraints, and iterative guidance.  Apart from that
author-directed drafting assistance, the ideas and arguments are due to the
authors.  The authors independently verified every argument in its final
form and accept full responsibility for the paper.

\end{document}